\documentclass[12pt, a4paper]{amsart}

\usepackage[hmargin=30mm, vmargin=25mm, includefoot, twoside]{geometry}
\usepackage{hyperref}

\usepackage{amsfonts,amssymb,verbatim}
\usepackage{latexsym}
\usepackage{mathrsfs}
\usepackage{stmaryrd}
\usepackage{xspace}
\usepackage{enumerate, paralist}
\usepackage{graphicx}
\usepackage[all]{xy}
\usepackage{extarrows}
\usepackage{tabu}

\usepackage[usenames,dvipsnames]{xcolor}

\usepackage{txfonts, pxfonts}

\usepackage{amsthm}
\usepackage{amsmath}

\usepackage[T1]{fontenc}	
\usepackage[utf8]{inputenc}			
\usepackage{mathtools} 					
\usepackage{microtype} 					
\usepackage{bbm}					
\usepackage{bm}						
\usepackage[shortcuts]{extdash} 			
\usepackage{subcaption}					
\usepackage{mparhack}					
\usepackage{accents}

\usepackage{todonotes}
		\newcommand{\NN}{\mathbb{N}}	
			
		\newcommand{\RR}{\mathbb{R}}

		\newcommand{\ZZ}{\mathbb{Z}}

		\newcommand{\CF}{\mathcal{F}}	
\newcommand{\CG}{\mathcal{G}}

		\newcommand{\CP}{\mathcal{P}}

\newcommand{\ubar}[1]{\underaccent{\bar}{#1}} 		

\newcommand{\bs}{\backslash} 				
\DeclarePairedDelimiter\abs{\lvert}{\rvert}		
\DeclarePairedDelimiter\norm{\lVert}{\rVert}		
\DeclarePairedDelimiter\angles{\langle}{\rangle}	
\DeclarePairedDelimiter\paren{(}{)}			
\DeclarePairedDelimiter\brackets{[}{]}			
\DeclarePairedDelimiter\braces{\{}{\}}			

	\newcommand{\bigparen}[1]{\paren[\big]{#1}}
	\newcommand{\Bigparen}[1]{\paren[\Big]{#1}}
\newcommand{\bigbrack}[1]{\brackets[\big]{#1}}

	\newcommand{\bigbraces}[1]{\braces[\big]{#1}}
	\newcommand{\Bigbraces}[1]{\braces[\Big]{#1}}
\newcommand{\bigmid}{\mathrel{\big|}}			
\newcommand{\Bigmid}{\mathrel{\Big|}}			

\theoremstyle{plain}
\newtheorem{thm}{Theorem}[section]				
\newtheorem{prop}[thm]{Proposition}		
\newtheorem{lem}[thm]{Lemma}						
\newtheorem{cor}[thm]{Corollary}

\newtheorem*{thm*}{Theorem}			\newtheorem*{theorem*}{Theorem}		
\newtheorem*{prop*}{Proposition}		\newtheorem*{proposition*}{Proposition}
\newtheorem*{lem*}{Lemma}			\newtheorem*{lemma*}{Lemma}			
\newtheorem*{cor*}{Corollary}			\newtheorem*{corollary*}{Corollary}
\newtheorem*{qu*}{Question}			\newtheorem*{question*}{Question}
\newtheorem*{conj*}{Conjecture}			\newtheorem*{conjecture*}{Question}
\newtheorem*{fact*}{Fact}
\newtheorem*{claim*}{Claim}

\newtheorem{alphthm}{Theorem}			

\newtheorem{alphcor}[alphthm]{Corollary}

\theoremstyle{definition}
\newtheorem{de}[thm]{Definition}

\newtheorem*{de*}{Definition}			\newtheorem{definition*}{Definition}	
\newtheorem*{notation*}{Notation}	
\newtheorem*{conv*}{Convention}			\newtheorem*{convention*}{Convention}

\theoremstyle{remark}
\newtheorem{rmk}[thm]{Remark}						
\newtheorem{exmp}[thm]{Example}			\newtheorem{example}[thm]{Example}

\numberwithin{equation}{section}

 \theoremstyle{definition}
  \newtheorem{defn}[thm]{Definition}
  \newtheorem*{stdassump}{Standing assumption}
 \theoremstyle{remark}
 \newtheorem{rem}[thm]{Remark}

\def\diam{\mathrm{diam}}

\def\N{\mathbb N}

\newcommand{\acbdry}{\partial^{\Gamma}}

\newcommand{\fol}{F\o{}lner\xspace}
\newcommand{\fnc}{\ubar c}
\newcommand{\fnb}{\ubar b}
\newcommand{\fnk}{\ubar k}

\newcommand{\fnh}{\ubar h}

\newcommand{\mappgraph}[1]{\T^\nu\paren*{#1}}
\newcommand{\appgraph}[1]{\T\paren*{#1}}
\newcommand*{\mabs}[1]{\abs{#1}_\nu}
\newcommand{\symdiff}{{\, \resizebox{0.8 em}{!}{$\triangle$}\, }}

\newcommand\sout[1]{}

\def\P{\mathcal P}
\def\G{\mathcal G}
\def\F{\mathcal F}

\def\T{\mathcal T}

\def\d{\mathrm{d}}
 
\def\act{\curvearrowright}

\begin{document}

\title{Asymptotic expansion in measure and strong ergodicity}

\author{Kang Li}
\address{Institute of Mathematics of the Polish Academy of Sciences, ul.\ \'Sniadeckich 8, 00-656 Warsaw, Poland}
	\email{kli@impan.pl}

\author{Federico Vigolo}
\address{Faculty of Mathematics and Computer Science, Weizmann Institute, 7610001, Rehovot, Israel}
\email{federico.vigolo@weizmann.ac.il}
\author{Jiawen Zhang}
\address{School of Mathematics, University of Southampton, Highfield, SO17 1BJ, Southampton, United Kingdom}
\email{jiawen.zhang@soton.ac.uk}

\thanks{KL has received funding from the European Research Council (ERC) under the European Union's Horizon 2020 research and innovation programme (grant agreement no. 677120-INDEX)}
\thanks{FV was supported by the ISF Moked 713510 grant number 2919/19.}
\thanks{JZ was supported by the Sino-British Trust Fellowship by Royal Society, International Exchanges 2017 Cost Share (China) grant EC$\backslash$NSFC$\backslash$170341, and NSFC11871342.}

\begin{abstract}
In this paper, we introduce and study a notion of asymptotic expansion in measure for measurable actions. This generalises expansion in measure and provides a new perspective on the classical notion of strong ergodicity. Moreover, we obtain structure theorems for asymptotically expanding actions, showing that they admit exhaustions by domains of expansion. As an application, we recover a recent result of Marrakchi, characterising strong ergodicity in terms of local spectral gaps. We also show that homogeneous strongly ergodic actions are always expanding in measure and establish a connection between asymptotic expansion in measure and asymptotic expanders by means of approximating spaces.
\end{abstract}

\date{\today} 
\maketitle
\noindent\textit{Mathematics Subject Classification (2020): 37A30, 37A15. Secondary: 05C48.}\\
\textit{Keywords: asymptotic expansion in measure; asymptotic expanders; domain of expansion; strong ergodicity; local spectral gap.}\\

\section{Introduction}

This paper is the first half of a longer study on the dynamical analogue of a notion of ``asymptotic expansion'' for graphs. \emph{Asymptotic expanders} are families of finite metric spaces satisfying some weak expansion property. They were introduced in \cite{intro} as part of a study of $C^*$\=/algebras and projections associated with expander graphs.
In a sense that was made precise in \cite{structure}, asymptotic expanders are families of metric spaces that are ``almost everywhere expanders''.

It has long been known there are analogies between expansion for graphs and some dynamical properties of actions of finitely generated groups: this idea lies at the heart of many early constructions of explicit families of \emph{expander graphs} \cite{bourgain_expansion_2013, gabber_explicit_1981, Mar73, shalom_expanding_1997}. 
It was proved in \cite{Vig19} that, under relatively mild hypotheses, families of graphs approximating a measurable action $\rho\colon\Gamma\curvearrowright (X,\nu)$ are expanders if and only if $\rho$ is \emph{expanding in measure}. For measure\=/preserving actions, expansion in measure is also equivalent to a classical notion of \emph{spectral gap} (see \cite{Vig19}). The spectral characterisation of expansion in measure can be used to provide numerous explicit examples of expanding actions (\emph{e.g.}, using \cite{benoist_spectral_2014,bourgain_spectral_2007,conze_ergodicity_2013,GJS99}, \emph{etc.}) and hence to construct expander graphs via the approximation procedure of \cite{Vig19}.

This project began as an attempt to adapt the techniques from \cite{Vig19} to investigate dynamical analogues of asymptotic expansion. This led to the definition of \emph{asymptotic expansion in measure} (see below for a definition). Surprisingly, it turned out that---for measure\=/class\=/preserving actions---asymptotic expansion in measure is equivalent to yet another classical notion in dynamical systems: \emph{strong ergodicity}.
We could then translate the structure theory of asymptotic expanders developed in \cite{structure} to the context of strongly ergodic actions. Interestingly, the dynamical analogue of a key lemma in \cite{structure} turns out to be a version of a maximality trick that is often used in the von Neumann algebraic approach to dynamical systems (Section~\ref{ssec:maximal.folner}). Our approach led us to produce elementary proofs of a few recent results concerning strongly ergodic actions and to the discovery of some completely new phenomena (these new phenomena are mainly developed in \cite{dynamics2}). 

The main technical result of this paper is a structure theorem for actions that are asymptotically expanding in measure (Theorem~\ref{thm:intro:structure}). This theorem directly implies a result of Marrakchi \cite{Mar18}. Moreover, Theorem~\ref{thm:intro:structure} is used heavily in the sequel to this paper, where asymptotic expansion is studied in relation with Markov kernels, operators of finite propagation and Roe algebras of associated warped cones. This study eventually leads to a characterisation of strong ergodicity in terms of rank\=/one projections and to the production of new counterexamples to the coarse Baum--Connes conjecture \cite[Theorems E,F and G]{dynamics2}.

A second contribution of this paper is a satisfactory extension of the approximation procedure of \cite{Vig19} to the study of asymptotic expansion. This technique can be used to produce examples of asymptotic expanders, and it also served as a motivation and source of inspiration for all the results presented in this paper and its sequel \cite{dynamics2}.

We will now give a more detailed overview of this work.

\subsection{The structure theory for asymptotic expansion}

Recall that a measure-class-preserving action on a probability space $\Gamma \act (X,\nu)$ is \emph{strongly ergodic} if any sequence of almost invariant measurable subsets must be asymptotically trivial. That is,
any sequence $\{C_n\}_{n \in \N}$ of measurable subsets of $X$ such that $\lim_{n\to \infty} \nu(C_n \symdiff \gamma C_n)=0$ for every $\gamma \in \Gamma$ must satisfy $\lim_{n\to \infty} \nu(C_n)(1-\nu(C_n))=0$. 
This notion was introduced in \cite{connes_property_1980,schmidt1980asymptotically,Sch81} in relation with Ruziewicz problem, Kazhdan's Property~(T) and amenability. Since then, it became clear that strong ergodicity is also useful in the context of von Neumann algebras and it was recently used to prove rigidity results \cite{ioana2017strong}.
We refer to \cite{houdayer2017strongly,ozawa2016remark,vaes2018bernoulli} for interesting examples of strongly ergodic actions and further motivation.
For measure-preserving actions, spectral gap implies strong ergodicity. It was shown in \cite{Sch81} the converse is not true in general.

As already announced, this paper is concerned with the dynamical counterpart to asymptotic expanders (we refer to Section~\ref{sec:prelims} for the definition of asymptotic expanders). This notion is defined as follows:

\begin{defn}\label{def:intro:asymptotic expansion}
A measurable action $\rho\colon \Gamma\curvearrowright (X,\nu)$ of a countable discrete group on a probability space is called \emph{asymptotically expanding (in measure)} if there exist for every $0<\alpha\leq \frac 12$ a constant $c(\alpha)>0$ and a finite subset $S(\alpha)\subseteq \Gamma$ such that
 \[
  \nu\big(\bigcup_{s\in S(\alpha)} s\cdot A\big)>(1+c(\alpha))\nu(A)
 \]
for every measurable subset $A\subseteq X$ with $\alpha\leq\nu(A)\leq \frac 12$.
\end{defn}

Definition~\ref{def:intro:asymptotic expansion} is clearly a weakening of expansion in measure as defined in \cite{Vig19} (see also \cite{grabowski_measurable_2016}). Furthermore, it is not hard to see that if $\rho$ is measure-class-preserving then it is asymptotically expanding in measure \emph{if and only if} it is strongly ergodic (see Proposition~\ref{prop:strongly ergodic iff asymptotic expanding}).
More precisely, asymptotic expansion in measure can be regarded as a quantitative version of strong ergodicity, as the dependence of $c(\alpha)$ and $S(\alpha)$ on the parameter $\alpha$ in Definition \ref{def:intro:asymptotic expansion} describes quantitatively the triviality of any sequence of almost invariant measurable sets. 
Its quantitative nature will play a crucial role in the sequel to this work \cite{dynamics2}.

The following diagram summarizes the relations between the properties introduced thus far:
\\[-.3cm]
\begin{figure}[h!] 
\centerline{
\xymatrixcolsep{3.5pc}\xymatrixrowsep{2pc}
\xymatrix{
  \rm{Expansion~in~measure} \ar@2{<->}[r]^-{\textstyle \rm{m.p.}} \ar@{=>}[d] & \rm{Spectral~gap} \ar@{=>}[d]^-{\textstyle \rm{m.p.}}  \\
  \rm{Asymptotic~expansion~in~measure} \ar@2{<->}[r]^-{\textstyle \rm{m.c.p.}}
                 & \rm{Strong~ergodicity}.  
}}
\caption*{\footnotesize m.p.: measure-preserving, \quad m.c.p.: measure-class-preserving.}
\caption{Relations of actions on probability spaces}
\end{figure}

In our investigation, an important role will be played by the ``localised'' versions of the above mentioned properties. 
Some of these have already been defined. Namely, it has already been settled what it means that an action $\Gamma\curvearrowright (X,\nu)$ is expanding in measure or has a spectral gap ``restricted'' to some finite measure subset $Y\subseteq X$ (note that the set $Y$ is not required to be $\Gamma$\=/invariant). In the former case $Y$ is said to be a \emph{domain of expansion (in measure)} (\cite{grabowski_measurable_2016}, see also Definition \ref{defn:domain of expansion}), and in the latter the action is said to have a \emph{local spectral gap with respect to $Y$} (\cite{BIG17}, see also Definition \ref{defn:local spectral gap}). 

Although defined independently, these localised versions of expansion in measure and spectral gap remain equivalent in the measure-preserving case. That is, if $\Gamma\curvearrowright(X,\nu)$ is a measure-preserving action and $Y\subseteq X$ has finite measure, then it is a domain of expansion \emph{if and only if} the action has local spectral gap with respect to $Y$ (\cite[Lemma~5.2]{grabowski_measurable_2016}). This can be seen as an indication of the naturality of this approach. 
Following the same philosophy, we thus introduce the notion of \emph{domain of asymptotic expansion (in measure)} as a localised version of asymptotic expansion in measure:

\begin{defn}\label{def:intro:domain asymptotic expansion}
For a given measurable action of a countable discrete group $\Gamma$ on a measure space $(X,\nu)$, a positive finite measure subset $Y\subseteq X$ is a \emph{domain of asymptotic expansion (in measure)} if there exist for every $0<\alpha\leq \frac 12$ a constant $c(\alpha)>0$ and a finite subset $S(\alpha)\subseteq \Gamma$ such that
 \[
  \nu\big(\bigcup_{s\in S(\alpha)} (s\cdot A)\cap Y\big)>(1+c(\alpha))\nu(A)
 \]
for every measurable subset $A\subseteq Y$ with $\alpha\nu(Y)\leq\nu(A)\leq \frac 12\nu(Y)$.
\end{defn}

Recall that an \emph{exhaustion} of a measure space $X$ is an increasing sequence $X_n\subseteq X_{n+1}$ such that $X=\bigcup_{n\in\NN}X_n$ up to measure zero sets. Our main theorem can be summarised as follows:

\begin{alphthm}\label{thm:intro:structure}
 Let $\rho\colon\Gamma\curvearrowright (X,\nu)$ be a measure\=/class\=/preserving action of a countable discrete group $\Gamma$ on a $\sigma$\=/finite measure space $(X,\nu)$. Then the following are equivalent:
 \begin{enumerate}[(1)]
  \item $\rho$ is strongly ergodic;
  \item $X$ admits an exhaustion by domains of expansion; 
  \item every finite measure subset is a domain of asymptotic expansion;
  \item $\rho$ is ergodic and $X$ admits a domain of (asymptotic) expansion.
 \end{enumerate}
\end{alphthm}

The proof is fairly involved but makes only use of elementary measure theory. 
In the case of measure\=/preserving actions, the equivalence ``$(1)\Leftrightarrow(4)$'' yields an elementary proof of \cite[Theorem A]{Mar18}:

\begin{alphcor}
 A measure\=/preserving ergodic action $\rho\colon\Gamma\curvearrowright (X,\nu)$ of a countable discrete group $\Gamma$ on a $\sigma$\=/finite measure space $(X,\nu)$ is strongly ergodic \emph{if and only if} it has local spectral gap with respect to some $Y\subseteq X$ with $0<\nu(Y)<\infty$.
\end{alphcor}

As a matter of fact, our approach to Theorem \ref{thm:intro:structure} allows us to prove slightly more refined results. Firstly, if we assume that the measure $\nu$ is finite we can prove a similar result for more general (\emph{i.e.}, non necessarily measure\=/class preserving) actions. More precisely, for a probability space $(X,\nu)$ conditions $(2)$ and $(3)$ of Theorem~\ref{thm:intro:structure} are equivalent to asymptotic expansion in measure\footnote{%
In other words, we replace strong ergodicity in $(1)$ of Theorem~\ref{thm:intro:structure} with asymptotic expansion in measure. As already discussed, these two notions are equivalent for measure\=/class\=/preserving actions (see Proposition~\ref{prop:strongly ergodic iff asymptotic expanding}).
}
for \emph{any} measurable action $\Gamma\curvearrowright (X,\nu)$. 
We refer to Theorem~\ref{thm:structure theorem for probability} for a precise statement. It is worth pointing out that we first prove this finite\=/measure version of the theorem and then use it as a stepping stone to prove Theorem~\ref{thm:intro:structure} for general $\sigma$\=/finite measure spaces (Theorem~\ref{thm:structure theorem general}).

Secondly, we also obtain a more ``quantitative'' control on the expansion properties of the action. In particular, if the acting group $\Gamma$ is finitely generated and $S$ is a finite symmetric generating set, then Theorem~\ref{thm:intro:structure} remains true if one changes the definition of asymptotic expansion (Definitions~\ref{def:intro:asymptotic expansion} and~\ref{def:intro:domain asymptotic expansion}) by requiring that $S(\alpha)\equiv S$ independently of $\alpha$ (see Proposition \ref{prop:structure thm finitely generated case} and Corollary \ref{cor:quantitative version of structure thm for f.g. groups}). 

\

Theorem~\ref{thm:intro:structure} can be used to shed more light on the relation between strong ergodicity and spectral gap. This is best explained 
by examining Schmidt's example of strongly ergodic probability measure\=/preserving action that fails to have a spectral gap \cite{Sch81}.
In our terminology, his example is constructed as an action on an infinite union $X=\bigsqcup_{k\geq 0} X_k$ such that each finite union $\bigsqcup_{k=0}^n X_k$ is a domain of expansion, but the sequence of infinite tails $\bigsqcup_{k>n}X_k$ is a sequence of almost invariant subsets. It is interesting to note that all the ``non-expansion'' is confined in arbitrarily small tails in $X$. The construction of an exhaustion by domains of expansion in Theorem~\ref{thm:intro:structure} implies that this is indeed the general picture.
That is, if a probability measure\=/preserving action $\rho\colon\Gamma\curvearrowright(X,\nu)$ is strongly ergodic but fails to have a spectral gap, then all the non\=/expansion is concentrated in some (arbitrarily small) subset of $X$.  
It is reasonable to expect that the dynamics within a small set of non\=/expansion is quantitatively different from the rest, or in other words, that the action $\rho$ cannot be ``homogeneous''. In turn, if one knows that $\Gamma\curvearrowright (X,\nu)$ \emph{is} homogeneous and strongly ergodic then the action \emph{must} have a spectral gap. Other authors recently made similar observations (\cite[Theorem~4]{abert2012dynamical} and \cite[Lemma~10]{chifan2010ergodic}). 
The techniques developed in this paper allow us to make the above argument precise and to prove the following quantitative statement:

\begin{alphthm}
\label{thm:intro:homogeneous strong ergodicity}
Let $\rho\colon \Gamma\curvearrowright (X,\nu)$ be a measurable action of a countable discrete group on a probability space and assume that $\rho$ commutes with an ergodic measure\=/preserving action. If $\rho$ is asymptotically expanding ``at the scale of $\alpha_0$'' for some $0<\alpha_0\leq \frac 14$, then $\rho$ is expanding in measure.
\end{alphthm}

We refer the readers to Theorem~\ref{thm:homogeneous strong ergodic} for a precise statement. Since a strongly ergodic action on a probability space is asymptotically expanding in measure---and hence asymptotically expanding ``at every scale''---we immediately recover {\cite[Lemma~10]{chifan2010ergodic}} (and \cite[Theorem~4]{abert2012dynamical}):

\begin{alphcor}
Let $\rho: \Gamma\curvearrowright (X,\nu)$ be a probability measure preserving action of a countable discrete group that commutes with an ergodic measure\=/preserving action. Then $\rho$ is strongly ergodic \emph{if and only if} it has a spectral gap.
\end{alphcor}

We remark that Theorem~\ref{thm:homogeneous strong ergodic} is more precise than the previously known results in two respects: it applies to actions $\rho$ that need not preserve the measure (nor the measure\=/class), and its proof produces an explicit bound on the expansion constant in terms of the parameters of asymptotic expansion.

\subsection{Approximating spaces}
Last but not least, we conclude this introduction by illustrating the result that prompted us into this research project. That is, the relation between asymptotic expansion in measure and asymptotic expanders (Definition~\ref{defn:asymptotic.expanders}). As mentioned at the beginning of this introduction, this relation is obtained by extending the approximating procedure developed in \cite{Vig19}. Concretely, we prove the following:

\begin{alphthm}\label{thm:intro.approximating}
Let $(X,d,\nu)$ be a locally compact metric space with a Radon probability measure, and $\{\P_n\}_{n\in \NN}$ a sequence of measurable partitions of $X$ with uniformly bounded measure ratios and $\mathrm{mesh}(\CP_n)\to 0$. Then a continuous measure\=/class\=/preserving action $\Gamma \act X$ of a countable discrete group is asymptotically expanding in measure \emph{if and only if} the associated approximating spaces $\braces{\appgraph{\P_n}}_{n \in \N}$ are a sequence of asymptotic expanders.
\end{alphthm}

This result can be used as an efficient and unified way to construct asymptotic expanders. In turn, this has consequences with regard to the existence of coarse embeddings into Hilbert spaces and the coarse Baum--Connes conjecture \cite{structure,intro,dynamics2}. We refer the readers to Section~\ref{sec:measured.appgraphs.and.asymp.expansion} for relevant definitions and a detailed explanation of the approximation procedure. 

It is worth noting that---albeit similar in spirit---the exposition given in Section~\ref{sec:measured.appgraphs.and.asymp.expansion} differs significantly from the original treatment of \cite{Vig19}. This is mainly due to the fact that our proof Theorem~\ref{thm:intro.approximating} passes through the definition of yet another construct, which we call \emph{measured asymptotic expanders}. This intermediate object is rather convenient, as it allows us to move harmoniously between dynamical systems and metric spaces.

As a concluding remark, we wish to stress that Theorem~\ref{thm:intro.approximating} was for us a great heuristic tool. Both Theorem~\ref{thm:intro:structure} and Theorem~\ref{thm:intro:homogeneous strong ergodicity} were inspired from analogous results for asymptotic expanders \cite{structure}, and the same goes for \cite[Theorems D,E,F]{dynamics2}. Since asymptotic expanders are families of \emph{finite} metric spaces, we found that it is generally easier to prove results in this setting first and then adapt their proof to obtain analogous results for dynamical systems. The fact that our approach naturally led us to rediscover so many recent results for strongly ergodic actions can be seen as a testament to the validity of this heuristic. 
 
\

\emph{Organisation of the paper.}
In Section~\ref{sec:prelims}, we cover some general preliminaries and review some related literature. In Section~\ref{sec:asymptotic.expansion in measure}, we define asymptotic expansion in measure and its local version. We also prove a number of technical but useful lemmas that will be used throughout. Section~\ref{sec:structure} and Section~\ref{sec:homogeneity} are devoted to providing the proofs of Theorem~\ref{thm:intro:structure} and Theorem~\ref{thm:intro:homogeneous strong ergodicity}, respectively. In Section~\ref{sec:measured.appgraphs.and.asymp.expansion}, we establish the connection between asymptotically expanding actions and asymptotic expanders via approximating spaces and prove Theorem \ref{thm:intro.approximating}.

\subsection*{Acknowledgements}
We wish to thank Amine Marrakchi for pointing out \cite{chifan2010ergodic,houdayer2017strongly,Mar18} to us and for manifesting interest in our work.
The first and third authors would like to thank Piotr Nowak and J\'{a}n \v{S}pakula for introducing them to the topic of asymptotic expansion and for several early discussions.  Finally, the second author wishes to thank Uri Bader for his encouragement and helpful comments.

\section{Preliminaries}
\label{sec:prelims}

\subsection{Measure spaces}
Let $(X,\nu)$ be a measure space. 
When $\nu(X)<\infty$, we generally assume that it is a probability space. 
In the sequel we will frequently use Radon--Nikodym derivatives, we hence work with the following:

\begin{stdassump}
 All the measures are assumed to be $\sigma$\=/finite.
\end{stdassump}

We generally call a measurable subset $A\subseteq X$ of positive finite measure a \emph{domain}.

\begin{de}\label{defn:exhaustion}
 An \emph{exhaustion} of a measure space $(X,\nu)$ is a sequence of nested measurable subsets $Y_1\subseteq Y_2\subseteq\cdots $ such that $\bigcup_{n\in\NN}Y_n=X$ up to measure zero. We denote exhaustions by $Y_n\nearrow (X,\nu)$, or simply $Y_n\nearrow X$, if the measure is clear from the context. We will be especially interested in exhaustions where each $Y_n$ has positive finite measure (\emph{i.e.}, exhaustions by domains).
\end{de}

Note that if $Y_n\nearrow (X,\nu)$ is an exhaustion and $\nu'\sim\nu$ is an equivalent measure, then $Y_n\nearrow (X,\nu')$ is an exhaustion as well. When both $\nu$ and $\nu'$ are finite, the following holds true:

\begin{lem}\label{lem:infinitesimal sequences are preserved}
 Let $\nu$ and $\nu'$ be equivalent finite measures on a space $X$ and $(A_n)_{n\in\NN}$ a sequence of measurable subsets of $X$, then $\nu(A_n)\to 0$ \emph{if and only if} $\nu'(A_n)\to 0$.
\end{lem}

The following corollary will be used frequently in the sequel: 

\begin{cor}\label{cor:equivalent measures have uniform ratios}
 Given two equivalent finite measures $\nu$ and $\nu'$ on $X$, there exist increasing functions $\rho_-,\rho_+\colon (0,1)\to(0,1)$ such that 
 \[
  \rho_-(\alpha)\nu'(X)\leq\nu'(A)\leq\rho_+(\beta)\nu'(X)
 \]
 for every $\alpha,\beta\in(0,1)$ with $\alpha \leq \beta$, and measurable $A\subseteq X$ with $\alpha\nu(X)\leq\nu(A)\leq\beta\nu(X)$.
\end{cor}

Note that Lemma~\ref{lem:infinitesimal sequences are preserved} and Corollary \ref{cor:equivalent measures have uniform ratios} are trivially false for infinite measure spaces.

\subsection{Actions on measure spaces}\label{ssec:prelims:actions.on.measure.spaces}

Throughout the paper, $\Gamma$ will always denote a countable discrete group. 
In order to make some statements more quantitative and to highlight the relation between (asymptotic) expansion in measure and (asymptotic) expanders (Subsection~\ref{ssec:prelims:expanders}), we will always equip $\Gamma$ with a \emph{proper length function} $\ell: \Gamma \to \{0\}\cup\NN$, \emph{i.e.}, $\ell$ is a proper function and satisfies:
\begin{itemize}
  \item $\ell(\gamma)=0$ if and only if $\gamma=1$ (the identity element in $\Gamma$);
  \item $\ell(\gamma)=\ell(\gamma^{-1})$ for every $\gamma\in \Gamma$;
  \item $\ell(\gamma_1\gamma_2) \leq \ell(\gamma_1)+\ell(\gamma_2)$ for every $\gamma_1,\gamma_2\in \Gamma$.
\end{itemize}

It is easy to see that every $\Gamma$ admits a proper length function (see e.g. \cite[Proposition~1.2.2]{NY12}). For example, if $\Gamma$ is finitely generated and $S$ is a finite generating set with $S=S^{-1}$, we can take the length function to be the \emph{word length}:
\[
\ell(\gamma) \coloneqq \min\big\{n\bigmid \gamma=s_{i_1}s_{i_2}\ldots s_{i_n} \mbox{~where~}s_{i_k}\in S\text{ for every }k=1,\ldots, n\big\}.
\]

Any proper length function $\ell$ induces a left-invariant metric $d_\ell$ on $\Gamma$ by $d_\ell(\gamma_1,\gamma_2) \coloneqq \ell(\gamma_1^{-1}\gamma_2)$. This makes $\Gamma$ into a \emph{proper} discrete metric space in the sense that every bounded subset is finite. For each $k\in \NN$, denote by $B_k$ the closed ball with radius $k$ and centre at the identity:
\[
B_k \coloneqq \{\gamma\in \Gamma \mid \ell(\gamma) \leq k\}.
\]
It follows directly from definition that each $B_k$ is finite and symmetric (i.e., $\gamma\in B_k$ implies that $\gamma^{-1} \in B_k$), $1\in B_k$ and $B_k \cdot B_l \subseteq B_{k+l}$ for any $k,l\in \NN$.

We will be concerned with measurable actions of $\Gamma$ on $(X,\nu)$. Given $A \subseteq X$ and $K \subseteq \Gamma$, let
\[
K \cdot A \coloneqq \bigcup_{\gamma\in K} \gamma \cdot A.
\] 
Note that $A\subseteq B_k\cdot A$ for every $A\subseteq X$ and every $k\in \N$. 

Recall that a measurable action $\Gamma\curvearrowright (X,\nu)$ is \emph{measure\=/class\=/preserving} if it sends measure\=/zero sets to measure\=/zero sets.
In particular, for every $\gamma\in\Gamma$ there is a Radon--Nikodym derivative $\d \gamma_*\nu/\d\nu$ that is well\=/defined up to measure-zero sets.

\subsection{Expansion in measure}

 Let $\Gamma\curvearrowright X$ be a measurable action and $S\subseteq\Gamma$ a finite symmetric set. For any measurable subset $A\subseteq X$ we define by $\acbdry_S A\coloneqq S\cdot A\smallsetminus A$. Given $k\in \NN$, we let $\acbdry_k A\coloneqq \acbdry_{B_k}A$.
This piece of notation is useful and rather suggestive.
In fact, one should think of $\acbdry_S A$ as the ``boundary of $A$ with respect to the action by $S$''.

\begin{de}[\cite{Vig19}]\label{defn:expanding in measure}
 A measurable action $\rho\colon\Gamma\curvearrowright(X,\nu)$ of a countable discrete group $\Gamma$ on a probability measure space $(X,\nu)$ is called \emph{expanding in measure} if there exist constants $c>0$ and $k\in \NN$ such that for any measurable subset $A \subseteq X$ with $0<\nu(A) \leq \frac{1}{2}$, we have $\nu(\acbdry_k A)> c\nu(A)$. In this case, we say that $\rho$ is \emph{$(c,k)$\=/expanding in measure}.
\end{de}

\begin{rem}\label{rmk:expansion.equiv.} 
 Given a measurable action $\rho\colon\Gamma\curvearrowright(X,\nu)$, it is easy to see that expansion in measure for $\rho$ is equivalent to any of the following:
 \begin{enumerate}
  \item there exists $c>0$ and a finite subset $S\subseteq \Gamma$ such that $\nu(S\cdot A)\geq (1+c) \nu(A)$ for any measurable subset $A \subseteq X$ with $0<\nu(A) \leq \frac{1}{2}$;
  \item there exists a finitely generated subgroup $\Gamma' \leq \Gamma$ such that the restriction $\rho'\colon\Gamma'\curvearrowright(X,\nu)$ is expanding in measure;
  \item for every $c\in (0,1)$ there exists a $k\in \NN$ such that $\rho$ is $(c,k)$\=/expanding in measure (see also Corollary \ref{cor:expansion.equiv.}).
 \end{enumerate}
In particular, expansion in measure does not depend on the choice of proper length function.
 If $\Gamma$ is finitely generated and $\ell$ is a word length, this is also equivalent to:
 \begin{itemize}
  \item[(4)] there exists $c>0$ such that $\rho$ is $(c,1)$\=/expanding in measure (this is the definition used in \cite[Definition 3.1]{Vig19}).
 \end{itemize}
\end{rem}

As already mentioned, we decide to fix a length function and use the nomenclature ``$(c,k)$\=/expansion'' to stress the geometric nature of this notion. On the other hand, it is sometimes convenient to highlight the presence of a finite set in $\Gamma$ that witnesses expansion in measure (this is the approach taken Section 1, see \emph{e.g.}, Definition~\ref{def:intro:asymptotic expansion}). For this reason, we will also say that an action is \emph{$(c,S)$\=/expanding in measure} (or simply \emph{$S$\=/expanding in measure}) if $S\subset \Gamma$ is a finite symmetric subset.
We will also use the same convention for the local and asymptotic versions of expansion in measure (see below).

In an independent work, Grabowski--Máthé--Pikurko defined a ``local'' version of expansion in measure under the name of \emph{domain of expansion}: 

\begin{de}[\cite{grabowski_measurable_2016}]\label{defn:domain of expansion} 
Let $\rho\colon \Gamma\curvearrowright (X,\nu)$ be a measurable action of a countable discrete group $\Gamma$ on a measure space $(X,\nu)$. A positive finite measure subset $Y \subseteq X$ is called a \emph{domain of expansion} for $\rho$ if there exist constants $c>0$ and $k\in \NN$ such that for every measurable subset $A \subseteq Y$ with $0<\nu(A) \leq \frac{\nu(Y)}{2}$, we have
 \[
 \nu\bigparen{(B_k \cdot A)\cap Y} > \paren{1+c}\nu(A).
\]
In this case, we say that $Y$ is a \emph{domain of $(c,k)$\=/expansion}.

Given a finite symmetric $S\subseteq\Gamma$, we say that $Y$ is a \emph{domain of $(c,S)$\=/expansion} if for every measurable subset $A \subseteq X$ with $0< \nu(A) \leq \frac{\nu(X)}{2}$, we have $ \nu\bigparen{(S \cdot A)\cap Y} > \paren{1+c}\nu(A)$. We also say that $Y$ is \emph{a domain of $S$\=/expansion} if it is a domain of $(c,S)$\=/expansion for some constant $c>0$.
\end{de}

Note that when $\nu$ is finite, $\Gamma\curvearrowright (X,\nu)$ is expanding in measure \emph{if and only if} $X$ is a domain of expansion.

\begin{rmk}\label{rmk:domain of expansion grabowski}
It is easy to verify that our definition of domain of expansion is
equivalent to the one given in \cite{grabowski_measurable_2016} (the authors of \cite{grabowski_measurable_2016} only consider measure preserving actions, but their definition makes sense also for general measurable actions).
\end{rmk}

\subsection{(Local) spectral gap}
A measure\=/preserving action $\rho\colon\Gamma\curvearrowright(X,\nu)$ of a countable discrete group $\Gamma$ on a probability measure space $(X,\nu)$ has a \emph{spectral gap} if there exist constants $\kappa>0$ and $k\in \NN$ such that for every function $f\in L^2(X;\nu)$ with $\int_Xf\d\nu=0$ we have
\[
 \norm{f}_2\leq \kappa\sum_{\gamma\in B_k}\norm{\gamma\cdot f-f}_2,
\]
where $\gamma\cdot f(x)\coloneqq f(\gamma^{-1} \cdot x)$.
It can be shown that the action $\rho$ is expanding in measure \emph{if and only if} it has a spectral gap (see, \emph{e.g.}, \cite[Section 7]{Vig19}). Note that this characterisation of expansion in terms of a spectral condition only holds for measure\=/preserving actions.

\begin{rmk}
 It follows from \cite[Theorem 3.2]{houdayer2017strongly} that the characterisation of expansion in measure in terms of spectral gap also holds for action that do not preserve the measure as long as they have bounded Radon--Nikodym derivatives. We will not need this fact in this paper.
\end{rmk}

In \cite{BIG17}, Boutonnet--Ioana--Golsefidy introduced the following localised version of spectral gap:

\begin{defn}[{\cite[Definition 1.2]{BIG17}}]\label{defn:local spectral gap}
Let $\rho\colon \Gamma \act (X,\nu)$ be a measure\=/preserving action of a countable discrete group $\Gamma$ on a measure space $(X,\nu)$, and $Y \subseteq X$ a measurable subset of positive finite measure.
The action $\rho$ has \emph{local spectral gap} with respect to $Y$ if there exist constants $\kappa>0$ and $k\in \NN$ such that
\[
\|f\|_{Y,2} \leq \kappa \sum_{\gamma \in B_k} \norm{\gamma\cdot f -f}_{Y,2}
\]
for any $f \in L^2(X;\nu)$ with $\int_Y f \mathrm{d}\nu=0$. Here $\|f\|_{Y,2}$ denotes the $L^2$-norm of the restriction of $f$ to $Y$:
\[
\|f\|_{Y,2} \coloneqq \big(\int_Y |f|^2 \mathrm{d}\nu\big)^{1/2}.
\]
\end{defn}
It is clear that when $\nu$ is a probability measure, $\rho$ has spectral gap if and only if it has local spectral gap with respect to the whole $X$. Also note that it is shown in \cite[Lemma 5.2]{grabowski_measurable_2016} that a measure\=/preserving action $\rho\colon \Gamma\curvearrowright (X,\nu)$ has local spectral gap with respect to a domain $Y\subseteq X$ \emph{if and only if} $Y$ is a domain of expansion for $\rho$. This fact can also be deduced from \cite[Theorem 3.2]{houdayer2017strongly} (or by adapting the arguments of \cite[Section 7]{Vig19}). A related but independent proof will appear in \cite{dynamics2} following a study of Markov kernels associated with measure\=/class\=/preserving actions.

\

These concepts turn out to be naturally related to the classical notion of strong ergodicity.

\begin{defn}[\cite{schmidt1980asymptotically}]\label{defn: strongly ergodic actions}
Let $\rho\colon \Gamma \act (X,\nu)$ be a measure\=/class\=/preserving action of a countable discrete group $\Gamma$ on a probability space $(X,\nu)$.
The action $\rho$ is called \emph{strongly ergodic} if every sequence of measurable subsets $\{C_n\}_{n \in \N}$ in $X$ such that $\lim_{n\to \infty} \nu(C_n \symdiff \gamma C_n)=0$ for every $\gamma \in \Gamma$, must satisfy
\[
\lim_{n\to \infty} \nu(C_n)(1-\nu(C_n))=0.
\]
\end{defn}

It follows directly from Lemma~\ref{lem:infinitesimal sequences are preserved} that strong ergodicity for an action on a probability space $(X,\nu)$ only depends on the measure\=/class\footnote{More generally, strong ergodicity is in fact invariant under orbit equivalence \cite{hjorth2005rigidity,Sch81}.}  of $\nu$. That is, if $\nu'$ is another probability measure equivalent to $\nu$ then $\rho$ is strongly ergodic with respect to $\nu$ \emph{if and only if} it is strongly ergodic with respect to $\nu'$. Hence, the following is well\=/posed:

\begin{de}[\cite{ioana2017strong}]\label{strong ergodic on infinite space}
 A measure\=/class\=/preserving action of a countable discrete group $\rho\colon\Gamma\curvearrowright (X,\nu)$ on a (possibly infinite but $\sigma$-finite) measure space is \emph{strongly ergodic} if $\rho\colon\Gamma\curvearrowright (X,\nu')$ is strongly ergodic with respect to some (hence every) probability measure $\nu'$ equivalent to $\nu$.
\end{de}

Finally, recall that it is recently proved in \cite[Theorem A]{Mar18} that an ergodic measure\=/preserving action on a $\sigma$\=/finite measure space $(X,\nu)$ is strongly ergodic \emph{if and only if} it has local spectral gap with respect to some $Y\subseteq X$ with $0<\nu(Y)<\infty$.

\subsection{Metric spaces, graphs and asymptotic expanders}\label{ssec:prelims:expanders}
Given a metric space $(X,d)$ and $r\geq 0$, we denote by $N_r(Y)\coloneqq\braces{x\in X\mid d(x,Y)\leq r}$ the \emph{closed $r$\=/neighbourhood} of a subset $Y\subseteq X$, and by $|Y|\in\NN\cup\{\infty\}$ the cardinality of $Y$.

\begin{defn}[{\cite[Definition~3.12]{intro}}]\label{defn:asymptotic.expanders}
 A sequence of finite metric spaces $\braces{X_n}_{n\in\NN}$ with $\abs{X_n}\to\infty$ is called a sequence of \emph{asymptotic expanders} if there are functions $\fnc\colon (0,\frac{1}{2}]\to \RR_{>0}$ and $\fnk\colon(0,\frac{1}{2}]\to \NN$ such that for every $\alpha\in(0,\frac{1}{2}]$, we have
 \[
  \abs{N_{\fnk(\alpha)}(A)} > \paren{1+\fnc(\alpha)}\abs{A}
 \]
 for every subset $A \subseteq X_n$ with $\alpha\abs{X_n}\leq \abs{A} \leq \frac{1}{2}\abs{X_n}$. In this case, we say that $\braces{X_n}_{n\in\NN}$ is a sequence of \emph{$(\fnc,\fnk)$\=/asymptotic expanders}. 
\end{defn}

 Just for reference, we recall that a connected graph is regarded as a discrete metric space by equipping its vertex set with the edge\=/path metric. 
 For any subset of vertices $A\subseteq \CG$, we denote by $\partial_1 A\coloneqq N_1(A)\smallsetminus A$ its $1$-boundary, and we recall that the \emph{(vertex) Cheeger constant} of a finite graph $\CG$ is defined as
 \[
  h(\CG)\coloneqq\inf\Bigbraces{\,\frac{\abs{\partial_1 A}}{\abs{A}} \bigmid A\subseteq \CG,\ 0<\abs{A}\leq \frac{1}{2}\abs{\CG}}.
 \]
 A sequence of \emph{expander graphs} is a sequence of finite graphs $\braces{\CG_n}_{n\in\NN}$ with uniformly bounded degree, cardinality going to infinity and a uniform lower bound on their Cheeger constants. Equivalently, finite graphs $\{\CG_n\}_n$ form a sequence of expanders \emph{if and only if} they have uniformly bounded degree and they are $(\fnc,1)$\=/asymptotic expanders for some constant function $\fnc$.

\section{Asymptotic expansion in measure}
\label{sec:asymptotic.expansion in measure}

In this section, we introduce the key concepts of this paper, namely the notion of ``asymptotic expansion in measure'' and its localised version. This naturally generalises the notion of expansion in measure and can be regarded as a dynamical analogue of asymptotic expansion for metric spaces (see Definition~\ref{defn:asymptotic.expanders}). We also explain how asymptotic expansion in measure relates with strong ergodicity.

\subsection{Asymptotically expanding actions and strong ergodicity}

We give the following:

\begin{defn}\label{defn:asymptotic expanding in measure}
Let $\rho\colon \Gamma \act (X,\nu)$ be a measurable action of a countable discrete group $\Gamma$ on a measure space $(X,\nu)$ with $\nu(X)<\infty$. The action $\rho$ is called \emph{asymptotically expanding (in measure)} if there exist functions $\fnc\colon (0,\frac{1}{2}]\to \RR_{>0}$ and $\fnk\colon(0,\frac{1}{2}]\to \NN$ such that for every $\alpha\in(0,\frac{1}{2}]$ we have
\begin{equation}\label{eq:def.asymptotic.expansion}
 \nu\bigparen{B_{\fnk(\alpha)} \cdot A} > \paren{1+\fnc(\alpha)}\nu(A)
\end{equation}
for every measurable subset $A \subseteq X$ with $\alpha \nu(X)\leq \nu(A) \leq \frac{\nu(X)}{2}$. In this case, we say that $\rho$ is \emph{$(\fnc,\fnk)$\=/asymptotically expanding (in measure)}. 

For a fixed finite symmetric $S\subset\Gamma$, we say that $\rho$ is \emph{$(\fnc,S)$\=/asymptotically expanding (in measure)} if for every $\alpha\in(0,\frac{1}{2}]$ and measurable subset $A \subseteq X$ with $\alpha \nu(X)\leq \nu(A) \leq \frac{\nu(X)}{2}$, we have $ \nu(S \cdot A) > \paren{1+\fnc(\alpha)}\nu(A)$. We also say that $\rho$ is \emph{$S$\=/asymptotically expanding (in measure)} if it is $(\fnc,S)$\=/asymptotically expanding for some function $\fnc$.
\end{defn}

The above definition generalises the notion of expansion in measure (see Definition~\ref{defn:expanding in measure}). More precisely, a measurable action is expanding in measure \emph{if and only if} it is $(\fnc, \fnk)$\=/asymptotically expanding for some constant functions $\fnc$ and $\fnk$.

\begin{rmk}
It is easy to see that asymptotic expansion in measure does not depend on the choice of proper length functions on the countable group $\Gamma$: choosing a different proper length function will merely yield different expansion functions $\fnc$ and $\fnk$. 
\end{rmk}

\begin{rmk} 
It would be possible to extend the notion of asymptotic expansion in measure to spaces of infinite measure following the same philosophy of \cite{Vig19}. Namely, by choosing functions $\fnc$ and $\fnk$ defined on $(0, \infty)$ and requiring that \eqref{eq:def.asymptotic.expansion} holds whenever $A\subset X$ has finite measure. However, this generalisation does not seem to capture particularly interesting phenomena.
\end{rmk}

\begin{rmk}\label{non-asy exp} 
Note that $\rho$ is \emph{not} asymptotically expanding in measure \emph{if and only if} there exist an $\alpha_0>0$ and a sequence of finite measure subsets $A_n\subseteq X$ with $\alpha_0\nu(X)\leq\nu(A_n)\leq\frac{1}{2}\nu(X)$, such that for every $n\in \NN$ we have
\[
 \nu\paren{B_n\cdot A_n}\leq \bigparen{1+\frac{1}{n}}\nu(A_n).
\]
\end{rmk}

The next result reveals the connection between asymptotic expansion in measure and strong ergodicity (see Definition \ref{defn: strongly ergodic actions}):

\begin{prop}\label{prop:strongly ergodic iff asymptotic expanding}
Let $\Gamma \act (X,\nu)$ be a measure\=/class\=/preserving action of a countable discrete group $\Gamma$ on a probability space $(X,\nu)$.
Then $\rho$ is strongly ergodic \emph{if and only if} it is asymptotically expanding in measure.
\end{prop}

\begin{proof}
\emph{Necessity}: Assume that $\rho$ is not asymptotically expanding, then there exist $\alpha_0\in (0,1/2)$ and a sequence of measurable subsets $\{A_n\}_{n \in \N}$ in $X$ with $\alpha_0 \leq \nu(A_n) \leq 1/2$ such that $\nu(B_n \cdot A_n)<(1+1/n)\nu(A_n)$ (see Remark~\ref{non-asy exp}). 

For any $\gamma \in \Gamma$, take a large $N \in \N$ such that $\gamma \in B_N$.
Hence for any $n>N$, we have $\nu(\gamma \cdot A_n \smallsetminus A_n) < \frac{1}{n} \nu(A_n)\leq \frac{1}{n}$. Similarly, we also have $\nu(\gamma^{-1} \cdot A_n \smallsetminus A_n) < \frac{1}{n}$. Since the action is measure\=/class\=/preserving, we deduce that 
\[
 \nu(A_n\smallsetminus \gamma\cdot A_n)=\nu(\gamma\cdot \paren{\gamma^{-1}\cdot A_n\smallsetminus A_n})\xrightarrow{\,n\to\infty\,}0,
\]
and hence
\[
\nu(\gamma \cdot A_n \symdiff A_n) = \nu(\gamma \cdot A_n \smallsetminus A_n) + \nu(A_n \smallsetminus \gamma \cdot A_n) \xrightarrow{\,n\to\infty\,}0.
\]
This is a contradiction to the fact that $\nu(A_n)(1-\nu(A_n)) \geq \alpha_0/2>0$. Hence $\rho$ is not strongly ergodic.

\emph{Sufficiency}: Assume that $\rho$ is not strongly ergodic, then there exist $\alpha_0>0$ and a sequence of measurable subsets $\{C_n\}_{n\in \N}$ with $\lim_{n\to \infty} \nu(C_n \symdiff \gamma C_n)=0$ for every $\gamma \in \Gamma$, while $\nu(C_n)(1-\nu(C_n))\geq \alpha_0$ for each $n$. Without loss of generality, we can assume that $\nu(C_n) \leq \frac{1}{2}$.
Hence for every $k\in \NN$, we have 
\[
\nu(B_k\cdot C_n)
 \leq \nu(C_n)+\sum_{\gamma\in B_k}\nu(\gamma C_n\smallsetminus C_n)
 \leq \nu(C_n) + \sum_{\gamma\in B_k}\nu(\gamma C_n\symdiff C_n)\to \nu(C_n),
\]
as $n \to \infty$. This is a contradiction to the assumption that $\rho$ is asymptotically expanding in measure.
\end{proof}

\begin{rmk}
 Classically, strong ergodicity is only defined for measure-class-preserving actions. In this paper we follows this convention and only refer to strong ergodicity when dealing with measure\=/class\=/preserving actions. On the other hand, Definition~\ref{defn: strongly ergodic actions} also makes formal sense for general measurable actions. The ``sufficiency'' part in the proof of Proposition~\ref{prop:strongly ergodic iff asymptotic expanding} works in this extended setting as well. This shows that asymptotic expansion in measure is a formally stronger property than strong ergodicity.
\end{rmk}

\subsection{Domains of asymptotic expansions}\label{ssec:domain of asymp exp}
We will now introduce the localised version of asymptotic expansion in measure, which is also the asymptotic version of Definition~\ref{defn:domain of expansion}. This notion will be crucial for the structure results in Section~\ref{sec:structure}.

\begin{de}\label{defn:domain of asymp expansion} 
Let $\rho\colon \Gamma\curvearrowright (X,\nu)$ be a measurable action of a countable discrete group $\Gamma$ on a measure space $(X,\nu)$. A positive finite measure subset $Y \subseteq X$ is called a \emph{domain of asymptotic expansion} for $\rho$ if there exist functions $\fnc\colon (0,\frac{1}{2}]\to \RR_{>0}$ and $\fnk\colon(0,\frac{1}{2}]\to \NN$ such that for every $\alpha\in(0,\frac{1}{2}]$ and measurable $A \subseteq Y$ with $\alpha\nu(Y)\leq \nu(A) \leq \frac{\nu(Y)}{2}$, we have
\[
 \nu\bigparen{(B_{\fnk(\alpha)} \cdot A)\cap Y} > \paren{1+\fnc(\alpha)}\nu(A).
\]
In this case, we say that $Y$ is a \emph{domain of $(\fnc,\fnk)$\=/asymptotic expansion}.

For a fixed finite symmetric $S\subset\Gamma$, we say that $Y$ is a \emph{domain of $(\fnc,S)$\=/asymptotic expansion} if for every $\alpha\in(0,\frac{1}{2}]$ and measurable subset $A \subseteq X$ with $\alpha \nu(X)\leq \nu(A) \leq \frac{\nu(X)}{2}$, we have $ \nu\bigparen{(S \cdot A)\cap Y} > \paren{1+\fnc(\alpha)}\nu(A)$. We also say that $Y$ is a \emph{domain of $S$\=/asymptotic expansion} if it is a domain of $(\fnc,S)$\=/asymptotic expansion for some function $\fnc$.
\end{de}

Obviously if $\nu(X)$ is finite, then $X$ is a domain of (asymptotic) expansion \emph{if and only if} $\rho$ is (asymptotically) expanding in measure. In particular, all the statements that are true for domains of (asymptotic) expansion are also true for actions that are (asymptotically) expanding in measure.

The following technical lemma is an elementary but useful observation that will allow us to deal more easily with subsets of measure greater than $\frac{1}{2}\nu(Y)$.

\begin{lem}\label{lem:annoying 1/2 upper bound}
Let  $Y\subseteq X$ be a domain of $(\fnc,\fnk)$\=/asymptotic expansion for a measurable action $\Gamma\curvearrowright (X,\nu)$ on a measure space $(X,\nu)$. Then there exist functions $\fnb\colon [\frac{1}{2},1)\to \RR_{>0}$ and $\fnh\colon[\frac{1}{2},1)\to \NN$ such that for every $\beta\in[\frac{1}{2},1)$, we have
\[
 \nu\bigparen{(B_{\fnh(\beta)} \cdot A)\cap Y} > \paren{1+\fnb(\beta)}\nu(A)
\]
 for every measurable subset $A\subseteq Y$ with $\frac{1}{2}\nu(Y)\leq \nu(A)\leq\beta \nu(Y)$. Furthermore, if $\fnk\equiv k$ is constant, then $\fnh$ can be chosen to be the same constant $k$.
\end{lem}

\begin{proof}
 Assume for simplicity that $\nu(Y)=1$. For every $\beta \in[\frac{1}{2},1)$, we let 
\begin{equation}\label{eq:hbar}
  \fnh(\beta)\coloneqq \fnk\Bigparen{\frac{1-\beta}{2}}.
\end{equation}
 Given any measurable subset $A\subseteq Y$ with $\frac{1}{2}\leq \nu(A)\leq \beta$, let $D\coloneqq Y\smallsetminus \paren{B_{\fnh(\beta)}\cdot A}$.
If $\nu(D)< \frac{1-\beta}{2}$, then $\nu\bigparen{(B_{\fnh(\beta)} \cdot A)\cap Y}> \frac{1+\beta}{2} \geq \frac{1+\beta}{2\beta}\nu(A)$. 

On the other hand if $\nu(D)\geq \frac{1-\beta}{2}$, then it follows from our choice of $\fnh(\beta)$ that $\nu\bigparen{(B_{\fnh(\beta)} \cdot D)\cap Y}>(1+c_\beta)\nu(D)$ where $c_\beta\coloneqq \fnc\bigparen{\frac{1-\beta}{2}}>0$. Since $B_{\fnh(\beta)}$ is symmetric, we have $B_{\fnh(\beta)}\cdot D\smallsetminus D\subseteq B_{\fnh(\beta)}\cdot A\smallsetminus A$ and it follows that
 \[
  \nu\bigparen{(B_{\fnh(\beta)} \cdot A \smallsetminus A)\cap Y} \geq  \nu\bigparen{(B_{\fnh(\beta)} \cdot D)\cap Y} - \nu(D)> c_\beta\nu(D)\geq c_\beta\frac{1-\beta}{2\beta}\nu(A).
 \]
Hence, letting $\fnb(\beta)\coloneqq\min\bigbraces{\frac{1-\beta}{2\beta}\,,\,c_\beta\frac{1-\beta}{2\beta}}$ yields the desired function $\fnb$. The ``Furthermore'' part is clear from \eqref{eq:hbar}.
\end{proof} 

The following lemma 
shows that the notion of asymptotic expansion in measure depends only on the measure\=/class of $\nu$:

\begin{lem}\label{lem:asymptotic expansion is invariant under equivalence}
 Let $\nu$ and $\nu'$ be two equivalent (possibly infinite) measures on $X$, and $\rho\colon\Gamma\curvearrowright X$ a measurable action. If $Y\subseteq X$ is a measurable subset of positive finite $\nu$ and $\nu'$\=/ measure, then it is a domain of asymptotic expansion with respect to $\nu$ \emph{if and only if} it is a domain of asymptotic expansion with respect to $\nu'$.
\end{lem}

\begin{proof}
 Up to rescaling, we can assume that $\nu(Y)=\nu'(Y)=1$. Assume that $Y$ is a domain of $(\fnc,\fnk)$\=/asymptotic expansion for $\nu$. For every fixed $\alpha'$, it follows from Corollary~\ref{cor:equivalent measures have uniform ratios} that there exist constants $0<\alpha\leq\frac{1}{2}<\beta<1$ such that $\alpha\leq\nu(A)\leq\beta$ for every measurable $A\subseteq Y$ with $\alpha'\leq\nu'(A)\leq\frac{1}{2}$.
 
Let $\fnb, \fnh$ be the functions obtained in Lemma~\ref{lem:annoying 1/2 upper bound}, and set $k'\coloneqq\fnk(\alpha)+\fnh(\beta)$ and $c'\coloneqq\min\braces{\fnc(\alpha),\fnb(\beta)}$. We obtain that for every measurable $A\subseteq Y$ with $\alpha'\leq\nu'(A)\leq\frac{1}{2}$:
 \[
  \nu\bigparen{(B_{k'}\cdot A \cap Y) \smallsetminus A}> c'\nu(A)\geq c'\alpha.
 \]
 We can then use Corollary~\ref{cor:equivalent measures have uniform ratios} once again to produce a uniform lower bound on $\nu'\paren{(B_{k'}\cdot A\cap Y)\smallsetminus A}$ as well.
\end{proof}

Asymptotic expansion in measure is also preserved under taking subsets:

\begin{lem}\label{lem:subsets of domains of as exp are domains}
 Let $\rho\colon \Gamma \act (X,\nu)$ be a measurable action and $Y\subseteq X$ a domain of asymptotic expansion. Then any measurable subset $Y'\subseteq Y$ with $\nu(Y')>0$ is a domain of asymptotic expansion.
\end{lem}

\begin{proof}
Without loss of generality, we can assume that $\nu(Y)=1$. Fixing $\alpha\in (0,\frac{\nu(Y')}{2}]$, we choose an $\epsilon$ with $\frac{\nu(Y')}{2}<\epsilon< \nu(Y')$ and let $\beta\coloneqq \nu(Y\smallsetminus Y')+\epsilon$.
If we can show that there exists a constant $k\in \NN$ such that for any measurable $A\subseteq Y'$ with $\alpha\nu(Y')\leq \nu(A)\leq\frac{\nu(Y')}{2}$, we have $\nu\bigparen{(B_{k}\cdot A)\cap Y} > \beta$, it follows that
\[
  \nu\bigparen{(B_{k}\cdot A)\cap Y'} 
  >\beta-\nu(Y\smallsetminus Y')
  =\epsilon\geq \frac{2\epsilon}{\nu(Y')}\nu(A)
  =\big( 1+\bigparen{\frac{2\epsilon}{\nu(Y')}-1} \big)\nu(A).
\]
Since we chose $\epsilon> \frac{\nu(Y')}{2}$, the constant $\frac{2\epsilon}{\nu(Y')}-1$ is positive and therefore the proof would be complete.

 It thus remains to find such a $k$. 
 By Lemma~\ref{lem:annoying 1/2 upper bound}, there exist $k_0 \in \NN$ and $b>0$ such that for any measurable $A'\subseteq Y$ with $\alpha \leq \nu(A') \leq \beta$, then 
\[
 \nu\bigparen{(B_{k_0}\cdot A')\cap Y} > (1+b)\nu(A') \geq (1+b)\alpha.
\]

 Let $A\subseteq Y'$ be any measurable subset with $\alpha \nu(Y') \leq \nu(A)\leq\frac{\nu(Y')}{2}$. If $\nu\bigparen{(B_{k_0}\cdot A)\cap Y} \leq \beta$, then we have:
\[
 \nu\bigparen{(B_{2k_0}\cdot A)\cap Y}
 \geq\nu\bigparen{\paren{B_{k_0}\cdot(B_{k_0}\cdot A\cap Y)}\cap Y}
 >(1+b)\nu\paren{(B_{k_0}\cdot A)\cap Y} \geq (1+b)^2 \alpha.
\]
Inductively, we see that if $\nu\bigparen{(B_{mk_0}\cdot A)\cap Y}\leq \beta$ then
\[
\nu\bigparen{(B_{(m+1)k_0}\cdot A)\cap Y} \geq (1+b)^{m+1} \alpha.
\]
Taking the logarithm, we find a $m_0\in\NN$ depending only on $\alpha, \beta$ and $b$ such that the constant $k\coloneqq m_0k_0$ satisfies our requirements.
\end{proof}

\begin{rmk}\label{robust}
Both Lemma \ref{lem:asymptotic expansion is invariant under equivalence} and \ref{lem:subsets of domains of as exp are domains} fail if one replaces ``domain of asymptotic expansion'' by ``domain of expansion''. This shows that \emph{asymptotic} expansion is more flexible and better\=/behaved under transformations than regular expansion.
\end{rmk}

The next lemma shows that for a domain of $(\fnc,\fnk)$\=/(asymptotic) expansion, by choosing a larger $\fnk$ it is always possible to require the function $\fnc$ to be an arbitrary constant. This will be used later in our structure result.

\begin{lem}\label{lem:equiv.defn.for.asymp.expansion.domain}
Let $\rho\colon \Gamma\curvearrowright (X,\nu)$ be a measurable action on a measure space $(X,\nu)$, and $Y\subseteq X$ a domain. Then:
\begin{enumerate}
  \item $Y$ is a domain of expansion \emph{if and only if} for every $c\in(0,1)$, there exists $k\in \NN$ such that for every measurable $A \subseteq Y$ with $0<\nu(A) \leq \frac{\nu(Y)}{2}$, we have
$\nu\bigparen{(B_k \cdot A)\cap Y} > \paren{1+c}\nu(A)$.
  \item $Y$ is a domain of asymptotic expansion \emph{if and only if} for every $c\in(0,1)$ and $\alpha\in (0,\frac{1}{2}]$, there exists $k_\alpha \in \NN$ such that for every measurable $A \subseteq Y$ with $\alpha\nu(Y)\leq\nu(A) \leq \frac{\nu(Y)}{2}$, we have $\nu\bigparen{(B_{k_\alpha} \cdot A)\cap Y} > \paren{1+c}\nu(A)$.
\end{enumerate}
\end{lem}

\begin{proof}
We only prove $(2)$, a similar proof will imply $(1)$ as well. Moreover, we only need to show the necessity since the sufficiency is trivial.
Without loss of generality, we can assume $\nu(Y)=1$. 

Fix $c\in (0,1)$ and $\alpha\in (0,\frac{1}{2}]$. By Lemma~\ref{lem:annoying 1/2 upper bound}, there exist constants $b>0$ and $h \in \NN$ such that for every finite measure subset $A \subseteq X$ with $\alpha\leq \nu(A) \leq \frac{1+c}{2}$, we have $\nu\bigparen{(B_{h} \cdot A)\cap Y} > (1+b)\nu(A)$.
Let $k\coloneqq mh$ where $m\coloneqq\lceil \log_{1+b}(1+c) \rceil$. Then either 
\[
 \nu\bigparen{(B_k \cdot A)\cap Y}> \frac{1+c}{2}\geq (1+c)\nu(A)
\]
or we deduce by induction on $m$ that 
\[
\nu\bigparen{(B_k \cdot A)\cap Y}
> (1+b)^m\nu(A)\geq (1+c)\nu(A).
\]
So the result holds.
\end{proof}

As a directly corollary of Lemma \ref{lem:equiv.defn.for.asymp.expansion.domain}(1), we recover the following result for expanding actions (which is already mentioned in Remark \ref{rmk:expansion.equiv.}(3)):
\begin{cor}\label{cor:expansion.equiv.}
A measurable action $\rho\colon \Gamma \act (X,\nu)$ of a countable discrete group $\Gamma$ on a probability measure space $(X,\nu)$ is expanding in measure \emph{if and only if} for any $c\in (0,1)$, there exists $k\in \NN$ such that for any measurable subset $A \subseteq X$ with $0<\nu(A) \leq \frac{1}{2}$, we have $\nu(\acbdry_k A)> c\nu(A)$.
\end{cor}

We conclude this subsection with the following fact showing that the union of two ``big'' domains of expansion is still a domain of expansion.
Using Lemma \ref{lem:annoying 1/2 upper bound} and  \ref{lem:equiv.defn.for.asymp.expansion.domain}, the proof is elementary and left to the reader.

\begin{lem}\label{lem:union of domains}
 Let $\rho\colon \Gamma \act (X,\nu)$ be a measurable action on a measure space $(X,\nu)$, and $Y \subseteq X$ a domain. Assume that $Y_1, Y_2\subseteq Y$ be domains of $(c,k)$\=/expansion for some $c>0$ and $k\in \NN$. If $\nu(Y_1)>\frac{3}{4}\nu(Y)$ and $\nu(Y_2)>\frac{3}{4}\nu(Y)$, then there exists $c'>0$ and $k' \in \NN$ such that $Y_1 \cup Y_2$ is a domain of both $(c,k')$\=/expansion and $(c',k)$\=/expansion.
\end{lem}

\subsection{The case of finitely generated groups}\label{ssec:finite generating case}
As shown in Lemma \ref{lem:equiv.defn.for.asymp.expansion.domain}(2), the parameter function $\fnc$ in the notion of asymptotic expansion in measure can always be made to be constant. The same cannot be said for the other parameter function $\fnk$. In other words, there exist asymptotically expanding actions that cannot be $S$\=/asymptotically expanding for any finite symmetric $S\subset \Gamma$:

\begin{example}
Fix a measurable action of the free group on two generators $\rho_0\colon F_2\curvearrowright (X,\nu)$ on a probability measure space $(X,\nu)$ that is expanding in measure. Let $F_\infty\coloneqq\angles{s_n\mid n\in\NN}$ be the free group on countably infinite generators and consider the action $\rho\colon F_\infty\curvearrowright \coprod_{n\geq 1}(X,2^{-n}\cdot\nu)$ defined as follows:
the first two generators $s_0,s_1$ act as $\rho_0$ on every copy of $X$; for every other $n \geq 2$, $s_n$ swaps the $n$\=/th copy of $X$ with the $(n-1)$\=/th copy and acts as the identity everywhere else. It is easy to see that $\rho$ is asymptotically expanding in measure, but for every finite $S\subseteq F_\infty$ the restriction of $\rho$ to the group generated by $S$ is not even ergodic.
\end{example}

However, if the group is finitely generated and the action is measure\=/class\=/preserving, it is possible to make the parameter function $\fnk$ constant. More precisely, it turns out that asymptotic expansion can always be recognised ``in one step'':

\begin{lem}\label{lem:finitely.generated asymp.exp.in.one.step}
 Let $\rho\colon \Gamma \act (X,\nu)$ be a measure\=/class\=/preserving action on a probability space $(X,\nu)$. If the group $\Gamma$ is finitely generated and $S=S^{-1}$ is a finite generating set, then $\rho$ is asymptotically expanding \emph{if and only if} it is $S$\=/asymptotically expanding.
\end{lem}

\begin{proof}
The sufficiency is trivial, so we only focus on the necessity. Assume that $\rho$ is $(\fnc,\fnk)$\=/asymptotically expanding for appropriate functions $\fnc,\fnk$. We need to show that for every fixed $\alpha>0$, there exists $c'>0$ such that $\nu\paren{\acbdry_S A}> c'\nu(A)$ for every measurable $A\subseteq X$ with $\alpha\leq\nu(A)\leq\frac{1}{2}$.

Let $k\coloneqq\fnk(\alpha)$ be fixed, and we also fix an $\epsilon>0$ such that $0<\epsilon<\alpha\fnc(\alpha)$. Since $\nu$ is a probability measure, $\rho$ is measure\=/class\=/preserving and $S$ is finite, it follows that there is a $D=D(\epsilon)$ large enough so that $\nu(S^{k-1}\cdot Z_D^{(k)})<\epsilon$, where $Z_D^{(k)}\subseteq X$ is the measurable set of points $x\in X$ satisfying $\d \gamma_*\nu/\d\nu(x)\geq D$ for some $\gamma\in S^{k-1}$. We fix such a constant $D>0$.

Let $A\subseteq X$ be a measurable subset with $\alpha\leq\nu(A)\leq\frac{1}{2}$, by hypothesis we have $\nu\paren{S^k\cdot A}> (1+\fnc(\alpha))\nu(A)$. Note that $\paren{S^k\cdot A}\smallsetminus A=\paren{S^{k-1}\cdot \acbdry_S A}\smallsetminus A$, hence we have
\[
 \nu\bigparen{(S^k\cdot A)\smallsetminus A}
 \leq \nu\paren{S^{k-1}\cdot \acbdry_S A}
 \leq \nu \bigparen{(S^{k-1}\cdot (\acbdry_S A\smallsetminus Z_D^{(k)})} + \nu(S^{k-1}\cdot Z_D^{(k)})
<\abs{S}^{k-1}D\nu\paren{\acbdry_S A}+\epsilon.
\]
Combining these inequalities, we obtain that
\[
 \nu\paren{\acbdry_S A}>\paren{\fnc(\alpha)\nu(A)-\epsilon}\abs{S}^{1-k}D^{-1}
 \geq\frac{\fnc(\alpha)-\epsilon/\alpha}{\abs{S}^{k-1}D}\nu(A).
\]
Since $\epsilon<\alpha\fnc(\alpha)$, we finish the proof.
\end{proof}

The above result is motivated by \cite[Lemma 2.8]{structure} and will have further applications in a subsequent work \cite{dynamics2}.

\begin{rmk}\label{rmk:k-step.asymp.exp vs 1-step.asymp.exp}
Note that the ``asymptoticity'' is crucial in the above result: the analogous statement for expanding actions is false (it would only be true under the assumption that the Radon--Nikodym derivatives be bounded). Also note that Lemma~\ref{lem:finitely.generated asymp.exp.in.one.step} is trivially false for \emph{domains} of (asymptotic) expansion.
\end{rmk}

\section{Structure theorems for strongly ergodic actions}\label{sec:structure}

In this section, we introduce a structure theory for strongly ergodic actions. We divide it in three cases: actions on probability spaces, actions on general $\sigma$-finite measure spaces, and actions by finitely generated groups. 
Even if they are of the same spirit, these cases are treated separately because they differ significantly in the details.
Note that these structural results are dynamical analogues of a structure theory first developed for asymptotic expanders \cite{structure}.

\subsection{Maximal \fol sets}\label{ssec:maximal.folner}
 The proof of the structure results relies on the existence of ``maximal'' \fol sets. This is an adaptation of a technique introduced in \cite{structure} to study finite metric spaces. It is worth pointing out that the existence of such sets is less obvious in the current dynamical setting. Furthermore, we also need to prove a slightly more technical result which is necessary to deal with spaces of infinite measure.
 
 As Amine Marrakchi pointed out to us, it turns out that this sort of maximality arguments is used fairly often by von Neumann algebraists, and appear to go all the way back to von Neumann himself. In particular, the lemmas that we prove in this section are very similar to some results used in the proof of \cite[Theorem 3.5]{houdayer2017strongly}.

\begin{defn}\label{defn:Folner sets}
Let $\rho\colon \Gamma \act (X,\nu)$ be a measurable action of a countable discrete group $\Gamma$ on a measure space $(X,\nu)$, $Y\subseteq X$ a domain, $\epsilon>0$ and $k\in \NN$ be fixed. Given a measurable subset $Z\subseteq Y$, an \emph{$(\epsilon,k)$\=/\fol set in $Y$ relative to $Z$} is a measurable subset $A\subseteq Y\smallsetminus Z$ such that $\nu(A) \leq \frac{1}{2} \nu(Y)$ and $\nu((\acbdry_k A\cap Y)\smallsetminus Z) \leq \epsilon\nu(A)$. When $Z=\emptyset$, such an $A$ is also called an \emph{$(\epsilon,k)$\=/\fol set in $Y$}.
When $\nu$ is finite and the domain $Y$ is not specified, we assume that it is the whole of $X$.
\end{defn}

Let $Z\subseteq Y$ be a fixed measurable subset, $\epsilon>0$ and $k\in \NN$ be some constants. We consider the family of all $(\epsilon,k)$\=/\fol sets in $Y$ relative to $Z$:
\[
\F_{\epsilon,k}^Z\coloneqq \bigbraces{F \subseteq Y\smallsetminus Z \bigmid 0 \leq \nu(F) \leq \frac{1}{2} \nu(Y),\ \nu((\acbdry_k F\cap Y)\smallsetminus Z) \leq \epsilon\nu(F) }
\]
and its quotient $\CF_{\epsilon,k}^Z/\sim$, where $F\sim F'$ if they coincide up to null\=/sets. The inclusion induces a partial ordering on $\CF_{\epsilon,k}^Z/\sim$, where $[F]\sqsubseteq[F']$ if $F\subseteq F'$ up to null\=/sets. Note that $\CF_{\epsilon,k}^Z/\sim$ is always non\=/empty, as it contains the empty set.

\begin{lem}\label{lem:exists maximal Folner}
The partially ordered set $(\CF_{\epsilon,k}^Z/\sim,\sqsubseteq)$ has maximal elements.
\end{lem}

\begin{proof}
 Choose any $F_0\in\CF_{\epsilon,k}^Z$ and let $\beta_0\coloneqq\sup\{\nu(F)\mid F\in\CF_{\epsilon,k}^Z \mbox{~and~}\ F_0\subseteq F\}\leq \frac{\nu(Y)}{2}$. We recursively choose subsets $F_n\in \CF_{\epsilon,k}^Z$ such that $F_{n-1}\subseteq F_n$ and $\nu(F_{n})> \beta_{n-1}-\frac{1}{n}$,
 where $\beta_{n}\coloneqq\sup\{\nu(F)\mid F\in\CF_{\epsilon,k}^Z \mbox{~and~}\ F_n\subseteq F\}$ (notice that $\beta_{n+1}\leq \beta_{n}\leq \frac {\nu(Y)}{2}$).
 
 If we have $\beta_n=\nu(F_n)$ for some $n\in\NN$, we can stop the recursive process because it follows that $[F_n]$ is maximal in $\CF_{\epsilon,k}^Z/\sim$. Otherwise, we obtain an increasing sequence of measurable subsets $F_0\subseteq F_1\subseteq F_2\subseteq\cdots$. Let $F\coloneqq\bigcup_{n\in\NN}F_n$. Since $(\acbdry_k F_n\smallsetminus F)\subseteq (\acbdry_k F_{n+1}\smallsetminus F)$ for every $n\in\NN$, we know that $\acbdry_k F=\bigcup_{n\in\NN}(\acbdry_k F_n\smallsetminus F)$ is an increasing countable union. Hence we have
 \[
  \nu((\acbdry_k F\cap Y)\smallsetminus Z)
  =\lim_{n\to\infty}\nu((\acbdry_k F_n\cap Y)\smallsetminus F\smallsetminus Z)
  \leq\lim_{n\to\infty}\nu((\acbdry_k F_n\cap Y)\smallsetminus Z)
  \leq \lim_{n\to\infty}\epsilon\nu(F_n)=\epsilon\nu(F),
 \]
which implies that $F\in\CF_{\epsilon,k}^Z$. Moreover, $[F]$ is maximal in $\CF_{\epsilon,k}^Z/\sim$ because $\nu(F_n)$ approaches $\beta_n$ as $n$ grows to infinity and hence any strictly larger set cannot be \fol.
\end{proof}

With an abuse of notation, we will refer to $F\in \CF_{\epsilon,k}^Z$ as being a \emph{maximal $(\epsilon,k)$\=/\fol set (in $Y$ relative to $Z$)} if its equivalence class $[F]$ is maximal in $\CF_{\epsilon,k}^Z/\sim$.

\begin{rmk}
Note that if the action is not measure\=/class\=/preserving, the set $\CF_{\epsilon,k}^Z$ may not be closed under measure\=/zero perturbations. More precisely, it may well be that two sets coincide up to measure\=/zero but have largely different boundaries. 
\end{rmk}

The following lemma is simple but essential.

\begin{lem}\label{lem:complement of Folner set}
Let $\Gamma\curvearrowright (X,\nu)$ be a measurable action, $Z\subseteq Y\subseteq X$ domains, $\epsilon \in (0,1)$ and $k \in \NN$.
If $F_{\epsilon,k}\subseteq Y \smallsetminus Z$ is a maximal $(\epsilon,k)$\=/\fol set in $Y$ relative to $Z$, then for every measurable subset $A\subseteq Y\smallsetminus (F_{\epsilon,k}\cup Z)$ with $0<\nu(A)\leq \frac{1}{2}\nu(Y)-\nu(F_{\epsilon,k})$ we have
\[
 \nu\bigparen{(\acbdry_k A \cap Y)\smallsetminus (F_{\epsilon,k}\cup Z)} 
 > \epsilon\nu(A).
\]
\end{lem}

\begin{proof}
Given such an $A\subseteq Y\smallsetminus (F_{\epsilon,k}\cup Z)$, we have $\nu(F_{\epsilon,k})<\nu(A\sqcup F_{\epsilon,k})\leq \frac{1}{2}\nu(Y)$. Furthermore, we have
\[
 \paren{\acbdry_k(A\sqcup F_{\epsilon,k})\cap Y}\smallsetminus Z
 \subseteq\bigparen{(\acbdry_k A\cap Y)\smallsetminus (F_{\epsilon,k}\cup Z)}\cup\bigparen{(\acbdry_k F_{\epsilon,k}\cap Y)\smallsetminus Z}.
\]
By the maximality of $F_{\epsilon,k}$, it follows that
\[
 \epsilon\nu(A\sqcup F_{\epsilon,k})
 <\nu\bigparen{\paren{\acbdry_k(A\sqcup F_{\epsilon,k})\cap Y}\smallsetminus Z}
 \leq\nu\bigparen{(\acbdry_k A\cap Y)\smallsetminus (F_{\epsilon,k}\cup Z)}+\nu\bigparen{(\acbdry_k F_{\epsilon,k}\cap Y)\smallsetminus Z},
\]
which finishes the proof because $\nu\bigparen{(\acbdry_k (F_{\epsilon,k})\cap Y)\smallsetminus Z}\leq \epsilon\nu(F_{\epsilon,k})$ by the hypothesis.
\end{proof}

\subsection{Structure Theorem for probability spaces}

Before we state the structure theorem for action on probability spaces, let us introduce the following ``localised'' version as a main technical step:

\begin{prop}\label{prop:exhausting domains by domains}
 Let $\Gamma\curvearrowright (X,\nu)$ be a measurable action of a countable discrete group $\Gamma$ on a measure space $(X,\nu)$ and $Y\subseteq X$ a domain of asymptotic expansion. 
Let $c\in(0,1)$ be any fixed constant and $(Z_n)_{n\in\NN}$ a sequence of nested subsets of $Y$ with $\nu(Z_n)\to 0$.
 Then there exist $N_0\in \NN$, a sequence of natural numbers $(k_n)_{n>N_0}$ and an exhaustion $Y_n\nearrow Y$ by domains of $(c,k_n)$\=/expansion such that $Y_n\subseteq Y\smallsetminus Z_n$ for every $n>N_0$.
\end{prop}

\begin{proof}
Without loss of generality, we assume that $\nu(Y)=1$. 
Since $\nu(Z_n)\to 0$, there exists $N_0\in \NN$ such that for each $n>N_0$, we have
\[
\nu(Z_n) < \frac{c^2}{8(c+1)}.
\]
We also fix a sequence $(\alpha_n)_{n>N_0}$ such that
\[
 0< \alpha_n <\frac{2}{c}\nu(Z_n).
\]

By Lemma \ref{lem:equiv.defn.for.asymp.expansion.domain}(2), there exists a sequence of natural numbers $(k_n)_{n>N_0}$ such that for every measurable subset $A\subseteq Y$ with $\alpha_n \leq \nu(A) \leq \frac{1}{2}$, we have $\nu(\acbdry_{k_n} A \cap Y) > c\nu(A)$. 
By Lemma~\ref{lem:exists maximal Folner}, for every $n>N_0$ there is a maximal $(\frac{c}{2}, k_n)$\=/\fol set $F_n$ in $Y$ relative to $Z_n$. 
We claim that the sets $\bar Y_n\coloneqq Y\smallsetminus (Z_n\sqcup F_n)$ are domains of $(\frac{c}{2},k_n)$\=/expansion such that $\nu(\bar Y_n)\to 1$.

We begin by showing that for every $n> N_0$:
\begin{equation}\label{eq:F_n is small}
  \nu(F_n)<\frac{2}{c}\nu(Z_n),
\end{equation}
which implies that $\nu(\bar Y_n)\to 1$. Note that 
\begin{equation}\label{eq:boundary estimate}
 \nu(\acbdry_{k_n} F_n \cap Y)\leq \nu((\acbdry_{k_n} F_n \cap Y)\smallsetminus Z_n)+\nu(Z_n) \leq \frac{c}{2}\nu(F_n) + \nu(Z_n).
\end{equation}
By the choice of $k_n$, if we had $\nu(F_n)\geq\frac{2}{c}\nu(Z_n)$ then we would have 
$\nu(\acbdry_{k_n} F_n \cap Y) > c\nu(F_n)$. Combined with \eqref{eq:boundary estimate}, this would imply $\nu(F_n)< \frac{2}{c}\nu(Z_n)$, contradicting the assumption that $\nu(F_n)\geq\frac{2}{c}\nu(Z_n)$.

We now show that $\bar Y_n$ is a domain of $(\frac{c}{2},k_n)$\=/expansion for every $n>N_0$.
Fix any measurable $A_n\subseteq \bar Y_n$ with $0<\nu(A_n)\leq \frac{1}{2}\nu(\bar Y_n)$. If $\nu(A_n)\leq \frac{1}{2}-\nu(F_n)$, it follows from Lemma~\ref{lem:complement of Folner set} that $\nu(\acbdry_{k_n} A_n\cap \bar Y_n)>\frac{c}{2}\nu(A_n)$ and we are done. Assume now that $\nu(A_n)>\frac{1}{2}-\nu(F_n)$. By \eqref{eq:F_n is small}, $\frac{1}{2}-\nu(F_n) > \frac{1}{2} - \frac{2}{c}\nu(Z_n)$. Since $\nu(Z_n)\leq \frac{c^2}{8(c+1)}<\frac{c}{8}$, we also have $\frac{1}{2}-\frac{2}{c}\nu(Z_n) > \frac{2}{c}\nu(Z_n) \geq \alpha_n$.
Again, we deduce by the definition of $k_n$ that
\[
 \nu(\acbdry_{k_n} A_n\cap \bar Y_n)\geq \nu(\acbdry_{k_n} A_n \cap Y)-\nu(Z_n\sqcup F_n)
 > c\nu(A_n) - \nu(Z_n\sqcup F_n).
\]
Since $\nu(F_n) < \frac{2}{c}\nu(Z_n)$ and $\nu(Z_n) < \frac{c^2}{8(c+1)}$, elementary calculation shows that
\[
\nu(Z_n \sqcup F_n) =\nu(Z_n)+\nu(F_n)< \frac{c}{2}\big( \frac{1}{2}-\nu(F_n) \big) < \frac{c}{2} \nu(A_n),
\]
which implies that
\[
\nu(\acbdry_{k_n} A_n\cap \bar Y_n) > c\nu(A_n) - \frac{c}{2} \nu(A_n) = \frac{c}{2}\nu(A_n).
\]

We thus proved our claim. To have an exhaustion it remains to find an \emph{increasing} sequence of domains. Enlarging $N_0$ if necessary, we can assume that $\nu(\bar Y_n)>\frac{3}{4}$ for every $n>N_0$ and set $Y_n\coloneqq\bigcup_{p=N_0+1}^n \bar Y_p$. Finally, applying
Lemma~\ref{lem:union of domains} we can find $k'_n$ such that $Y_n$ is a domain of $(c,k'_n)$\=/expansion. 
\end{proof}

The following is the structure theorem for asymptotically expanding actions on probability spaces. When the action is measure\=/class-preserving (and thus strong ergodicity is defined), we obtain as a special case a structure result for strongly ergodic actions.

\begin{thm}\label{thm:structure theorem for probability}
Let $\rho\colon \Gamma \act (X,\nu)$ be a measurable action of a countable discrete group on a probability space $(X,\nu)$.
Then the following are equivalent:
\begin{enumerate}
  \item $\rho$ is asymptotically expanding in measure;
  \item every $Y\subseteq X$ of positive measure admits an exhaustion by domains of expansion; 
  \item $X$ admits an exhaustion by domains of expansion; 
  \item $X$ admits an exhaustion by domains of asymptotic expansion;
\end{enumerate}
If $\rho$ is measure\=/class\=/preserving, the above conditions are also equivalent to:
\begin{itemize}
 \item [(5)] $\rho$ is strongly ergodic.
\end{itemize} 
\end{thm}

\begin{proof}
Since any positive measure subset of a domain of asymptotic expansion is a domain of asymptotic expansion (Lemma~\ref{lem:subsets of domains of as exp are domains}),
$``(1)\Rightarrow(2)''$ follows from Proposition \ref{prop:exhausting domains by domains} by setting $Z_n=\emptyset$. $``(2)\Rightarrow(3)\Rightarrow(4)"$ are trivial, and $``(1)\Leftrightarrow(5)''$ for measure\=/class\=/preserving actions is the statement of Proposition~\ref{prop:strongly ergodic iff asymptotic expanding}.

$``(4)\Rightarrow(1)"$: 
Let $Y_n\nearrow X$ be an exhaustion by domains of asymptotic expansion.
If $\rho$ were not asymptotically expanding in measure, there would be an $\alpha_0\in (0,1/2)$, a sequence of measurable subsets $\{A_m\}_{m \in \N}$ in $X$ with $\alpha_0 \leq \nu(A_m) \leq 1/2$ such that $\nu(B_{m} \cdot A_m)<(1+1/m)\nu(A_m)$. On the other hand, if $n$ is large enough so that $\nu(Y_n)\geq 1-\frac{\alpha_0}{2}$, then $\frac{\alpha_0}{2}\leq\nu(A_m\cap Y_n)\leq \frac{1}{2-\alpha_0}\nu(Y_n)$ for every $m\in\NN$.

Fix such an $n$. Since $Y_n$ is a domain of asymptotic expansion, it follows from Lemma~\ref{lem:annoying 1/2 upper bound} that there is some $\epsilon>0$ and $k\in\NN$ such that
\[
 \nu((B_{k}\cdot A_m)\cap Y_n)> (1+\epsilon)\nu(A_m)
\]
for every $m$. For $m \gg 1$, this is a contradiction to $\nu(B_{m} \cdot A_m)<(1+1/m)\nu(A_m)$.
\end{proof}

As a direct corollary to Proposition \ref{prop:exhausting domains by domains} and Theorem \ref{thm:structure theorem for probability}, we obtain the following quantitative version:

\begin{cor}\label{cor:quantitative version of structure thm}
 Let $\Gamma\curvearrowright (X,\nu)$ be a measurable action of a countable discrete group $\Gamma$ on a probability space $(X,\nu)$. Then it is asymptotically expanding in measure \emph{if and only if} for every $c\in (0,1)$ there exist a sequence $(k_n)_{n\in \NN}$ and an exhaustion $Y_n\nearrow X$ by domains of $(c,k_n)$\=/expansion.
\end{cor}

\subsection{Structure Theorem for general measure spaces}
In this subsection, we prove a structure theorem for strongly ergodic actions on general measure spaces. 
As a consequence, 
we give a more direct characterisation of strong ergodicity and clarify its relation with the notion of local spectral gap by reproving a result of Marrakchi \cite{Mar18}. 

Let us start with the following auxiliary lemma, inspired by \cite[Remark 1.3 (4)]{BIG17}:

\begin{lem}\label{lem:union of translates of domains of expansion}
 Let $\rho\colon \Gamma\curvearrowright (X,\nu)$ be a measure\=/class\=/preserving action on a measure space $(X,\nu)$, and $Y\subseteq X$ a domain of asymptotic expansion. If $H\subseteq\Gamma$ is a finite subset such that $\nu(H\cdot Y)$ is finite, then $H\cdot Y$ is also a domain of asymptotic expansion.
\end{lem}

\begin{proof}
 Let $Y$ be a domain of $(\fnc,\fnk)$\=/asymptotic expansion and assume without loss of generality that $\nu(H\cdot Y)=1$. Fix any $\alpha_0\in (0,\frac{1}{2}]$. Since $\rho$ is measure\=/class\=/preserving and $H$ is finite, there is a constant $\alpha>0$ such that for every $\gamma\in H$ and $A\subseteq \gamma\cdot Y$ with $\nu(A)\geq\frac{\alpha_0}{\abs{H}}$, we have $\nu(\gamma^{-1} \cdot A)\geq\alpha$. 
 Furthermore, since $\nu(H\cdot Y)=1$ there exists a constant $\beta<\nu(Y)$ large enough so that for every $A\subseteq Y$ with $\nu(A)>\beta$ we have $\nu(H\cdot A)> \frac{2}{3}$.

Taking $k\in \NN$ large enough such that $H \subseteq B_k$, it follows by construction that for every measurable $A\subseteq H\cdot Y$ with $\alpha_0\leq \nu(A)\leq \frac{1}{2}$, we have that $(B_k\cdot A)\cap Y$ has measure at least $\alpha$. If $\nu\paren{(B_k\cdot A)\cap Y}>\beta$, then 
 \begin{equation}\label{eq:translates of domains I}
  \nu\bigparen{(B_k^2\cdot A)\cap (H\cdot Y)}
  \geq \nu\bigparen{H\cdot((B_k\cdot A)\cap Y)\cap (H\cdot Y)} 
  > \frac{2}{3}\nu(H\cdot Y)\geq \bigparen{1+\frac{1}{3}}\nu(A).
 \end{equation}
 Otherwise, set $A'\coloneqq (B_k\cdot A)\cap Y$ and note that $\alpha\leq\nu(A')\leq\beta$. It follows by Lemma~\ref{lem:annoying 1/2 upper bound} that there exist $k_1\in\NN$ and $b>0$ depending only on $\fnk,\fnc,\alpha$ and $\beta$ such that $\nu\paren{(B_{k_1}\cdot A')\cap Y}> (1+b)\nu\paren{A'} $. Since
\[
\bigparen{(B_{k_1}\cdot A')\cap Y} \smallsetminus A'
 \subseteq \bigparen{(B_{k_1}\cdot (B_{k}\cdot A))\cap Y} \smallsetminus A 
 \subseteq \bigparen{(B_{k_1+k}\cdot A)\cap (H\cdot Y)} \smallsetminus A, 
\]
we see that
 \begin{equation}\label{eq:translates of domains II}
  \nu\bigparen{(B_{k_1+k}\cdot A)\cap (H\cdot Y)}
  \geq \nu(A)+\nu\bigparen{((B_{k_1}\cdot A')\cap Y)\smallsetminus A'} 
  > \bigparen{1+2\alpha b}\nu(A).
 \end{equation}
 Combining \eqref{eq:translates of domains I} and \eqref{eq:translates of domains II} proves the lemma.
\end{proof}

The following is the structure theorem for strongly ergodic actions on general measure spaces:

\begin{thm}\label{thm:structure theorem general}
 Let $\rho\colon\Gamma\curvearrowright (X,\nu)$ be a measure\=/class\=/preserving action of a countable discrete group $\Gamma$ on a $\sigma$\=/finite measure space $(X,\nu)$. The following are equivalent:
 \begin{enumerate}[(1)]
  \item $\rho$ is strongly ergodic;
  \item $X$ admits an exhaustion by domains of expansion; 
  \item $X$ admits an exhaustion by domains of asymptotic expansion;
  \item every finite measure subset is a domain of asymptotic expansion;
  \item every finite measure subset admits an exhaustion by domains of expansion; 
  \item $\rho$ is ergodic and $X$ admits a domain of expansion;
  \item $\rho$ is ergodic and $X$ admits a domain of asymptotic expansion.
 \end{enumerate}
\end{thm}

\begin{proof} 
 $``(1)\Rightarrow(2)"$: Let $\nu'$ be a fixed probability measure on $X$ equivalent to $\nu$, then $\rho\colon\Gamma\curvearrowright (X,\nu')$ is asymptotically expanding in measure by Proposition~\ref{prop:strongly ergodic iff asymptotic expanding}.
 For any $n\in\NN$, let $X_n\subseteq X$ be the set of points where the Radon--Nikodym derivative $\frac{\d \nu}{\d \nu'}$ is bounded between $\frac{1}{n}$ and $n$:
 \[
  X_n\coloneqq
  \Bigbraces{x\in X\Bigmid\frac{1}{n}\leq\frac{\d \nu}{\d \nu'}(x)\leq n}.
 \]
 Then $X_n\nearrow X$ is an exhaustion of $X$ by subsets of finite measure (with respect to both $\nu$ and $\nu'$). Since $\nu'$ is finite, the sets $Z_n\coloneqq X\smallsetminus X_n$ satisfy $\nu'(Z_n)\to 0$.  
 By Proposition~\ref{prop:exhausting domains by domains}, there exist an $N_0\in \NN$ and an exhaustion $Y_n\nearrow X$ with $n>N_0$ by domains of expansion for $\rho\colon\Gamma\curvearrowright (X,\nu')$ such that $Y_n\cap Z_n=\emptyset$ for all $n>N_0$. We claim that the sets $Y_n$ are domains of expansion for $\rho\colon\Gamma\curvearrowright (X,\nu)$ as well.
 
 Fix $n>N_0$ and let $A\subseteq Y_n$ be a measurable subset with $0<\nu(A)\leq\frac{1}{2}\nu(Y_n)$. Since $Y_n \subseteq X_n$, it follows that 
\[
\nu'(Y_n\smallsetminus A) \geq \frac{1}{n}\nu(Y_n \smallsetminus A) \geq \frac{1}{2n} \nu(Y_n) \geq \frac{1}{2n^2}\nu'(Y_n),
\]
which implies that $\nu'(A)\leq (1-\frac{1}{2n^2})\nu'(Y_n)$. Applying Lemma~\ref{lem:annoying 1/2 upper bound}, there exist constants $b>0$ and $k\in \NN$ independent of $A$ such that $\nu'(\acbdry_k A\cap Y_n)> b\nu'(A)$. Since the Radon--Nikodym derivative is uniformly bounded on $Y_n$, we deduce that
 \[
  \nu(\acbdry_k A\cap Y_n)> \frac{b}{n^2}\nu(A).
 \]
Hence $Y_n$ is a domain of expansion for $\rho\colon\Gamma\curvearrowright (X,\nu)$ as well.
 
 $``(2)\Rightarrow (3)"$ is obvious.
 
 $``(3)\Rightarrow (1)"$: Let $\nu'$ be a probability measure on $X$ equivalent to $\nu$, and $Y_n\nearrow X$ be an exhaustion by domains of asymptotic expansion. Then $Y_n\nearrow X$ is an exhaustion with respect to $\nu'$ as well. Furthermore, we deduce by Lemma~\ref{lem:asymptotic expansion is invariant under equivalence} that $Y_n\nearrow X$ is actually an exhaustion by domains of asymptotic expansion for $\nu'$. It then follows from Theorem~\ref{thm:structure theorem for probability} that $\rho\colon\Gamma\curvearrowright(X,\nu')$ is asymptotically expanding in measure, hence is strongly ergodic by Proposition~\ref{prop:strongly ergodic iff asymptotic expanding}. So condition (1) holds.

 $``(1)\Rightarrow (4)"$: Let $Y\subseteq X$ be any finite measure subset, and choose a probability measure $\nu'$ that is equivalent to $\nu$. By Proposition~\ref{prop:strongly ergodic iff asymptotic expanding}, $\Gamma\curvearrowright (X,\nu')$ is asymptotically expanding in measure and hence $(Y,\nu')$ is a domain of asymptotic expansion by Lemma~\ref{lem:subsets of domains of as exp are domains}. It follows from Lemma~\ref{lem:asymptotic expansion is invariant under equivalence} that $(Y,\nu)$ is a domain of asymptotic expansion as well.
 
 $``(4)\Rightarrow(5)"$ follows from Proposition~\ref{prop:exhausting domains by domains}.

 $``(5)\Rightarrow(6)"$: It suffices to show that the action $\rho$ is ergodic. If not, then there exists a non\=/trivial decomposition of $X$ into two $\Gamma$\=/invariant measurable subsets $X_1$ and $X_2$. Choose a finite measure subset $Y\subseteq X$ with $\nu(Y\cap X_1)>0$ and $\nu(Y\cap X_2)>0$. Condition (5) implies that $Y$ admits an exhaustion by domains of expansion. Hence there exists a domain $Y' \subseteq Y$ of expansion such that $\nu(Y'\cap X_1)>0$ and $\nu(Y'\cap X_2)>0$, and without loss of generality we can assume that $\nu(Y' \cap X_1) \leq \frac{1}{2}\nu(Y')$. However, note that for any subset $K \subseteq \Gamma$, we have 
 $\big(K \cdot (Y' \cap X_1) \smallsetminus (Y' \cap X_1) \big) \cap Y' = \emptyset$ since $X_1$ is $\Gamma$-invariant. Hence $Y'$ cannot be a domain of expansion, which is a contradiction.
 
 $``(6)\Rightarrow (7)"$ is obvious.
 
 $``(7)\Rightarrow (3)"$: As $n$ grows, the sets $B_n\cdot Y$ are an exhaustion of $(X,\nu)$ because $\rho$ is ergodic. Since the action is measure\=/class\=/preserving, we can find an exhaustion $Y'_n\nearrow Y$ such that $\nu(B_n\cdot Y'_n)$ is finite for every $n$ and $(B_n\cdot Y'_n)\nearrow (X,\nu)$. By Lemma~\ref{lem:subsets of domains of as exp are domains}, each $Y'_n$ is a domain of asymptotic expansion, and hence $B_n\cdot Y'_n$ is a domain of asymptotic expansion by Lemma~\ref{lem:union of translates of domains of expansion}.
\end{proof}

Note that the proof for $``(1)\Rightarrow (2)"$ of Theorem~\ref{thm:structure theorem general} actually shows that strong ergodicity implies that every measurable subset admits an exhaustion by domains of expansion (even if it has infinite measure). This can also be deduced from (5) with a simple diagonal argument.

Also note that (just as in Corollary~\ref{cor:quantitative version of structure thm}) Lemma~\ref{lem:equiv.defn.for.asymp.expansion.domain} implies that all the domains in Theorem~\ref{thm:structure theorem general} can be assumed to be of $(c,\fnk)$\=/(asymptotic) expansion for some fixed constant $c\in(0,1)$. However, the parameter function $\fnk$ will generally depend on the specific domain and the choice of $c$.

Finally, we record that the implication $``(6)\Rightarrow (1)"$ of Theorem~\ref{thm:structure theorem general} is a direct generalisation of the implication $``(4)\Rightarrow (5)"$ of \cite[Theorem 7.6]{BIG17}, while $``(1)\Leftrightarrow (6)"$ is a generalisation of \cite[Theorem A]{Mar18}. More precisely, recall that $Y\subseteq X$ is a domain of expansion for a measure\=/preserving action \emph{if and only if} the action has local spectral gap with respect to $Y$ (\cite[Lemma 5.2]{grabowski_measurable_2016}). Hence, Theorem~\ref{thm:structure theorem general} implies the following:

\begin{cor}[{\cite[Theorem A]{Mar18}}]\label{cor:strongly ergodic iff loc spec gap}
 A measure\=/preserving ergodic action $\rho\colon\Gamma\curvearrowright (X,\nu)$ of a countable discrete group $\Gamma$ on a $\sigma$\=/finite measure space $(X,\nu)$ is strongly ergodic \emph{if and only if} it has local spectral gap with respect to a domain $Y\subseteq X$.
\end{cor}

\subsection{Structure Theorem for actions by finitely generated groups}\label{ssec:structure for finite generated}
As we showed in Section \ref{ssec:finite generating case}, when groups are finitely generated the notion of asymptotic expansion in measure can be recognised ``in one step'' (Lemma \ref{lem:finitely.generated asymp.exp.in.one.step}). This observation leads to another version of the structure theorem for probability spaces. Intuitively speaking, the objective of this subsection is to gain control on the parameter $\fnk$ at the cost of the control we have on $\fnc$.

The key technical result is the following analogue of Proposition \ref{prop:exhausting domains by domains}: 

\begin{prop}\label{prop:structure thm finitely generated case}
Let $\Gamma\curvearrowright (X,\nu)$ be a measurable action of a countable discrete group $\Gamma$ on a measure space $(X,\nu)$ and $Y\subseteq X$ a domain of $(\fnc,S)$\=/asymptotic expansion for some finite symmetric subset $S\subseteq \Gamma$\footnote{Note that $S$ need not generate the group $\Gamma$ and we generally do not require $\Gamma$ to be finitely generated.}. Let $(Z_n)_{n\in\NN}$ be a sequence of nested subsets of $Y$ with $\nu(Z_n)\to 0$.
 Then there exist $N_0\in \NN$, a sequence $(c_n)_{n>N_0}$ and an exhaustion $Y_n\nearrow Y$ by domains of $(c_n,S)$\=/expansion such that $Y_n\subseteq Y\smallsetminus Z_n$ for every $n>N_0$.
\end{prop}

\begin{proof}
Without loss of generality, we assume that $\nu(Y)=1$. 
For $m\geq 2$, let $c_m\coloneqq \fnc(\frac{1}{m})$ and let $N_0$ be large enough so that $\nu(Z_n)\leq \frac{c_2}{4}$ for every $n>N_0$. Since $\nu(Z_n)\to 0$, it is possible to choose for each $n>N_0$ a natural number $m(n)\geq 2$ in such a way that the sequence $m(n)$ is non\=/decreasing, $m(n)\to\infty$, and 
\[
 \nu(Z_n)\leq \frac{c_{m(n)}}{2m(n)}
\]
for every $n>N_0$.

For every $n\geq N_0$, let $F_n$ be a maximal $(\frac{c_{m(n)}}{2},S)$\=/\fol set in $Y$ relative to $Z_n$. 
We claim that the sets $\bar Y_n\coloneqq Y\smallsetminus (Z_n\sqcup F_n)$ are domains of $S$\=/expansion such that $\nu(\bar Y_n)\to 1$. Note that 
\begin{equation}\label{boundary estimate 2}
 \nu(\acbdry_S F_n \cap Y)\leq \nu((\acbdry_S F_n \cap Y)\smallsetminus Z_n)+\nu(Z_n)\leq \bigparen{\frac{\nu(F_n)}{2}+\frac{1}{2m(n)} }c_{m(n)}.
\end{equation}
If $\nu(F_n)\geq \frac{1}{m(n)}$, then by the hypothesis we have $\nu(\acbdry_S F_n \cap Y) > c_{m(n)} \nu(F_n)$. Together with \eqref{boundary estimate 2} this implies that $\nu(F_n)< \frac{1}{m(n)}$, which is a contradiction. Hence we must always have $\nu(F_n)< \frac{1}{m(n)}$. In particular, it follows that $\nu(\bar Y_n)\to1$.

To show that $\bar Y_n$ is a domain of $S$\=/expansion, we fix any $A_n\subseteq \bar Y_n$ with $0< \nu(A_n)\leq \frac{1}{2}\nu(\bar Y_n)$. If $\nu(A_n)\leq \frac{1}{2}-\nu(F_n)$, it follows from Lemma~\ref{lem:complement of Folner set} that $\nu(\acbdry_S A_n\cap \bar Y_n)>\frac{c_{m(n)}}{2}\nu(A_n)$ and we are done. If $\frac{1}{2}-\nu(F_n)<\nu(A_n)$, then $\frac{1}{2}-\frac{1}{m(n)}<\nu(A_n)\leq \frac{1}{2}$. Thus we have:
\[
 \nu(\acbdry_S A_n\cap \bar Y_n)\geq \nu(\acbdry_S A_n \cap Y)-\nu(Z_n\sqcup F_n)
 > \fnc\bigparen{\frac{1}{2}-\frac{1}{m(n)}}-\bigbrack{\frac{1}{m(n)}+\frac{c_{m(n)}}{2m(n)}}.
\]
As $n$ goes to infinity $\fnc\bigparen{\frac{1}{2}-\frac{1}{m(n)}}$ stays bounded away from zero, while $\frac{1}{m(n)}+\frac{c_{m(n)}}{2m(n)}\to 0$. Enlarging $N_0$ if necessary, we can find $c'>0$ such that for any $n> N_0$ we have $\nu(\acbdry_S A_n\cap \bar Y_n)> c'\geq 2c'\nu(A_n)$. Combining the above two cases, we obtain that $\bar Y_n$ is a domain of $S$-expansion for every $n > N_0$. Finally, we can obtain an exhaustion by applying Lemma \ref{lem:union of domains}. This completes the proof.
\end{proof}

It is worthwhile pointing out that in some instances it is very useful to upgrade $(\fnc,\fnk)$\=/asymptotic expansion to $(\fnc',S)$\=/asymptotic expansion. Proposition~\ref{prop:structure thm finitely generated case} is an important tool in our forthcoming work \cite{dynamics2}. 

Combining Lemma \ref{lem:finitely.generated asymp.exp.in.one.step}, Theorem \ref{thm:structure theorem for probability} and Proposition \ref{prop:structure thm finitely generated case}, we obtain a second quantitative version of the structure result for finitely generated group actions (compare with Corollary \ref{cor:quantitative version of structure thm}):

\begin{cor}\label{cor:quantitative version of structure thm for f.g. groups}
 Let $\Gamma=\angles{S}$ be a group generated by a finite symmetric set $S$ and let $\rho\colon\Gamma\curvearrowright (X,\nu)$ be a measure\=/class\=/preserving action on a probability space $(X,\nu)$. Then $\rho$ is asymptotically expanding in measure \emph{if and only if}  there exists an exhaustion $Y_n\nearrow X$ by domains of $S$\=/expansion.
\end{cor}

\section{Strong ergodicity in homogeneous dynamical systems}\label{sec:homogeneity}

The aim of this section is to show that for ``sufficiently homogeneous'' actions, strong ergodicity is equivalent to expansion in measure. More precisely, recall that an \emph{automorphism} of a group action $\rho\colon\Gamma\curvearrowright (X,\nu)$ is a measure-preserving transformation $\lambda\colon (X,\nu)\to (X,\nu)$ that commutes with $\rho$. For our purposes, the relevant notion of ``sufficiently homogeneous'' is that the group of automorphisms of $\rho$ acts ergodically on $(X,\nu)$. In particular, the results of this section will apply to actions that commute with some measure\=/preserving ergodic action.

The heuristic reason to consider the above assumption
is provided by the structure theorem. More precisely if the action $\rho\colon\Gamma\curvearrowright (X,\nu)$ is strongly ergodic but not expanding in measure, it follows from Theorem~\ref{thm:structure theorem for probability} that there exists an exhaustion by domains of expansion $Y_n\nearrow X$. Hence any sequence of measurable subsets $A_k$ with $\nu(A_k)\leq \frac{1}{2}$ and $\nu(A_k \symdiff \gamma A_k)\to 0$ cannot be contained into any of the $Y_n$. However, using the ergodic action $(X,\nu)\curvearrowleft\Lambda$ we can ``almost'' transfer $A_k$ back into $Y_n$. Since the $\Lambda$-action is measure-preserving and commutes with $\rho$, these new sets will satisfy the same properties and we should hence obtain a contradiction. 
We should remark that, to a varying extent, this phenomenon was already observed by other authors (\emph{e.g.}, \cite[Theorem 4]{abert2012dynamical} and \cite[Lemma 10]{chifan2010ergodic}). In particular, the main result of this section can be seen as an effective version of \cite[Lemma 10]{chifan2010ergodic}:

\begin{thm}\label{thm:non.homogeneous}
\label{thm:homogeneous strong ergodic}
 Let $\rho\colon\Gamma\curvearrowright (X,\nu)$ be a measurable action of a countable discrete group $\Gamma$ on a probability space $(X,\nu)$, and $(X,\nu)\curvearrowleft \Lambda$ an ergodic measure\=/preserving action of a group $\Lambda$ which commutes with $\rho$. If there exist $0<\alpha_0\leq \frac 14$, $c>0$ and a finite symmetric subset $S\subseteq \Gamma$ such that 
 \[
  \nu(\acbdry_SA)> c\nu(A)
 \]
 for every $A\subseteq X$ with $\alpha_0\leq\nu(A)\leq\frac 12$, then $\Gamma\curvearrowright (X,\nu)$ is $S$\=/expanding in measure.
\end{thm}

\noindent\emph{Proof.}
For each $\alpha \in (0,\frac{1}{4}]$, let 
\begin{equation}\label{eq:fnc.as.infimum}
 \fnc(\alpha)\coloneqq\inf \bigbraces{\frac{\nu(\acbdry_S A)}{\nu(A)}\bigmid A\subseteq X,\ \alpha\leq\nu(A)\leq \frac 12}. 
\end{equation}
By assumption, we have $\fnc(\alpha_0)\geq c$. We will show that there exists an $\epsilon>0$ (depending only on $\alpha_0$) such that $\fnc(\alpha)\geq \epsilon\fnc(\alpha_0)$ for every $0<\alpha\leq\frac 12$. This implies that the $\Gamma$\=/action is $(\epsilon'c,S)$\=/expanding in measure for every $\epsilon'<\epsilon$. 
Since $\fnc$ is a non\=/decreasing function, it is enough to investigate its behaviour as $\alpha\to 0$. We will show that existence of the $\Lambda$\=/action implies that $\fnc(\alpha)$ cannot decrease too quickly as $\alpha$ goes to zero. 

We begin by noting that
\[
 \max_{0<\alpha\leq\frac 14}\frac{1-\sqrt{1-\alpha}}{\alpha}=4-2\sqrt{3}<1
\]
and we fix a $\delta>0$ such that $4-2\sqrt{3}<\delta<1$. 
We are now going to produce a lower bound on $\fnc(\delta\alpha)$ in terms of $\fnc(\alpha)$. 

Fix $0<\alpha\leq \frac 14$ and consider a measurable subset $A\subseteq X$ with $\delta\alpha\leq\nu(A)\leq\frac 14$. Since $(X,\nu)\curvearrowleft\Lambda$ is ergodic and measure\=/preserving, it follows that
\[
 \inf_{\lambda\in\Lambda} \nu(A\cap (A\cdot \lambda))\leq \nu(A)^2
\]
and hence
\begin{equation}\label{eq:ergodicity.small.interesection}
 \sup_{\lambda\in\Lambda} \nu(A\cup (A\cdot \lambda))= \sup_{\lambda\in\Lambda} \bigbrack{2\nu(A)-\nu(A\cap(A\cdot \lambda))}
 \geq 2\nu(A)-\nu(A)^2.
\end{equation}
The equation $\nu(A)^2-2\nu(A)+\alpha=0$ has solutions $1\pm\sqrt{1-\alpha}$. By the choice of $\delta$ it follows that $1-\sqrt{1-\alpha}<\delta\alpha\leq\nu(A)\leq\frac 14< 1+\sqrt{1-\alpha}$ and hence we obtain that $2\nu(A)-\nu(A)^2 > \alpha$.

Since the actions $\Gamma\curvearrowright X\curvearrowleft \Lambda$ commute, we see that $\acbdry_S(A\cdot \lambda)=(\acbdry_SA)\cdot \lambda$. Since $\nu$ is $\Lambda$\=/invariant, it follows that $\acbdry_S(A\cdot \lambda)$ and $\acbdry_S A$ have the same measure. Note that for every pair of measurable sets $C_1,C_2\subseteq X$ we have
\[
 \nu\paren{\acbdry_S(C_1\cup C_2)}\leq \nu(\acbdry_S C_1)+\nu(\acbdry_S C_2)-\nu(\acbdry_S (C_1\cap C_2)).
\]
In particular, for every $\delta\alpha\leq\nu(A)\leq\frac 14$ we have
\begin{align*}
 2\nu(\acbdry_S A)
 &\geq \sup_{\lambda\in\Lambda}\bigbrack{ \nu(\acbdry_S A)+ \nu(\acbdry_S (A\cdot \lambda))-\nu\bigparen{\acbdry_S (A\cap (A\cdot \lambda))} } \\
 &\geq \sup_{\lambda\in\Lambda}\nu\bigparen{\acbdry_S (A\cup (A\cdot \lambda))} 
 \geq \sup_{\lambda\in\Lambda} \fnc(\alpha)\nu(A\cup(A\cdot\lambda)) \\
 &\geq \fnc(\alpha)\nu(A)(2-\nu(A)),
\end{align*}
where we used \eqref{eq:ergodicity.small.interesection} for the last two inequalities. Note that the second\=/to\=/last inequality holds because $\nu(A\cup(A\cdot\lambda))\leq \frac 12$: this is where we need the assumption $\nu(A)\leq\frac 14$.

Now, either $\fnc(\delta\alpha)=\fnc(\alpha)$ or the infimum \eqref{eq:fnc.as.infimum} in the definition of $\fnc(\delta\alpha)$ can be approached by those $A\subseteq X$ with $\delta\alpha\leq\nu(A)<\alpha$. In either case, it follows from the above argument that 
\begin{equation}\label{eq:homogeneous recursion inequality}
 \fnc(\delta\alpha)\geq\bigparen{1-\frac{\alpha}{2}}\fnc(\alpha)
\end{equation}
for every $0<\alpha\leq\frac 14$. 
In turn, we can recursively obtain a bound on $\fnc(\delta^n\alpha_0)$ in terms of $\fnc(\alpha)$.

To complete the proof of Theorem~\ref{thm:non.homogeneous} it is enough to find a lower bound for the sequence $(\fnc(\delta^n\alpha_0))_{n\in \NN}$. A recursive application of \eqref{eq:homogeneous recursion inequality} reduces the problem to some elementary computations. In particular, we can complete the proof by applying the following lemma with constants $C\coloneqq\frac{ \alpha_0}{2}$ and $a_0\coloneqq \fnc(\alpha_0)$. \qed

\begin{lem}\label{lem:converging.sequence}
Let $\delta<1$, $0<C<1$, and $a_0>0$. Then there exists $\Phi>0$ depending only on $\delta $ and $C$ such that the recursively defined sequence
 \[
  a_{n+1}\coloneqq\paren{1-\delta^n C}a_n
 \]
 satisfies $a_n\geq \Phi a_0$ for every $n\in\NN$.
\end{lem}

\begin{proof}[Proof of Lemma~\ref{lem:converging.sequence}]
 For every $n\geq 1$ we see that
 \[
  a_n= a_0\prod_{k=1}^n \paren{1-\delta^kC}
 \]
 and therefore 
 \[
  \log(a_n)= \log(a_0) + \sum_{k=1}^n \log\paren{1-\delta^kC}.
 \]
 By the Taylor expansion of the logarithm,
 \[
  \sum_{k=1}^n \log\paren{1-\delta^kC}
  =\sum_{k=1}^n\sum_{j\geq 1}-\frac{(\delta^kC)^j}{j}
  > \sum_{k\in\NN}\sum_{j\geq 1}-(\delta^kC)^j
  = \sum_{j\geq 1}-\frac{C^j}{1-\delta^j}
 \]
 and the latter is bounded below from $-(1-\delta)^{-1}(1-C)^{-1}$. The claim follows by letting $\Phi\coloneqq\exp(-(1-\delta)^{-1}(1-C)^{-1})$.
\end{proof}

\begin{rmk} 
Our proof of Theorem~\ref{thm:homogeneous strong ergodic} is inspired by analogous results in the setting of graphs \cite[Section 7]{structure}, where it is proved that sequences of (bounded degree) vertex\=/transitive graphs are asymptotic expanders \emph{if and only if} they are genuine expanders. In the current setting, the existence of an ergodic commuting measure-preserving action is the dynamical analogue of vertex\=/transitivity. 

The discretisation procedure introduced in \cite{Vig19} provides more intuition for this analogy. Automorphisms of an action $\Gamma\curvearrowright(X,\nu)$ induce (coarse) automorphisms of the associated approximating graphs (see Definition \ref{defn:approximating graph} below), and ergodicity of the automorphism group implies (coarse) transitivity for said graphs. On the other hand, it is usually the case that asymptotic expansion (\emph{resp.} expansion) for actions is equivalent to asymptotic expansion (\emph{resp.} expansion) for the approximating graphs (see Section~\ref{sec:measured.appgraphs.and.asymp.expansion}). Using this correspondence, Theorem~\ref{thm:homogeneous strong ergodic} should be regarded as an analogue of \cite[Theorem 7.2]{structure}.

However, we should point out that Theorem~\ref{thm:homogeneous strong ergodic} does not follow from the results in \cite{structure}. More precisely, the fact that measurable automorphisms give rise to \emph{coarse} automorphisms of graphs turns out to be a substantial issue, which prevents us from applying tools in \cite{structure} directly. This forced us to use a more refined (and considerably more involved) argument than that of \cite[Theorem 7.2]{structure}.
\end{rmk}

Combining Proposition \ref{prop:strongly ergodic iff asymptotic expanding} with Theorem \ref{thm:homogeneous strong ergodic}, we recover the following:
\begin{cor}[{\cite[Lemma 10]{chifan2010ergodic}}]\label{cor:commute}
 Let $\rho\colon\Gamma\curvearrowright (X,\nu)$ be a measure\=/preserving action of a countable discrete group $\Gamma$ on a probability space $(X,\nu)$, and $(X,\nu)\curvearrowleft \Lambda$ an ergodic measure\=/preserving action of a group $\Lambda$ which commutes with $\rho$. Then $\rho$ is strongly ergodic \emph{if and only if} it is expanding in measure.
\end{cor}

As already noted in \cite{chifan2010ergodic}, the above corollary generalises {\cite[Proposition 3.1]{abert2012dynamical}} (which is the key ingredient in the proof of \cite[Theorem 4]{abert2012dynamical}):

\begin{cor}[{\cite[Proposition 3.1]{abert2012dynamical}}]
 Let $G$ be a compact group with a Haar measure $m$ and $\Gamma< G$ a countable subgroup. Then the action by left\=/multiplication $\Gamma\curvearrowright (G,m)$ is strongly ergodic \emph{if and only if} it has a spectral gap.
\end{cor}
\begin{proof}
 The action of right multiplication $(G,m)\curvearrowleft G$ is measure\=/preserving, ergodic and  commutes with left multiplications.
\end{proof}

\section{(Asymptotic) expansion on measured metric spaces and approximations}
\label{sec:measured.appgraphs.and.asymp.expansion}

The aim of this section is to extend the discretisation technique introduced in \cite{Vig19} and show that the notion of asymptotic expansion in measure is indeed the dynamical analogue of asymptotic expansion for sequences of finite metric spaces as defined in \cite{intro}.

\subsection{Measured asymptotic expanders}
Instead of considering ordinary asymptotic expanders (Definition~\ref{defn:asymptotic.expanders}), it will be more natural to work on their measured counterparts.

By a \emph{probability metric space} $\T^\nu$ we simply mean a finite or countably infinite discrete metric space $\T$ equipped with a probability Borel measure $\nu$ on $\T$ (note that we usually omit the metric in the notation). For simplicity, we also write $|A|_\nu \coloneqq \nu(A)$ for $A \subseteq \T^\nu$.

As a measured version for asymptotic expanders (see Definition \ref{defn:asymptotic.expanders}), we introduce the following:
\begin{de}
 A probability metric space $\T^\nu$ is called a \emph{measured asymptotic expander} if there are functions $\fnc\colon(0,\frac{1}{2}]\to \RR_{>0}$ and $\fnk\colon(0,\frac{1}{2}]\to \NN$ such that for every $\alpha\in(0,\frac{1}{2}]$, we have
\[
 \mabs{N_{\fnk(\alpha)}(A)} > \paren{1+\fnc(\alpha)}\mabs{A}
\]
for every subset $A \subseteq \T$ with $\alpha\leq \mabs{A} \leq \frac{1}{2}$. In this case, we say that $\T^\nu$ is a \emph{measured $(\fnc,\fnk)$\=/asymptotic expander}. 

A sequence of probability metric spaces $\{\T_n^{\nu_n}\}_{n\in \NN}$ is called a sequence of \emph{measured asymptotic expanders} if there are functions $\fnc\colon(0,\frac{1}{2}]\to \RR_{>0}$ and $\fnk\colon(0,\frac{1}{2}]\to \NN$ such that each $\T_n^{\nu_n}$ is a measured $(\fnc,\fnk)$\=/asymptotic expander. In this case, we say that $\{\T_n^{\nu_n}\}_{n\in \NN}$ is a sequence of \emph{measured $(\fnc,\fnk)$\=/asymptotic expanders}.
\end{de}

\begin{rmk}
 The above definition makes sense also for general (non necessarily countable and discrete) probability metric spaces. We do not need this level of generality in this work.
\end{rmk}

We collect in the following lemma a few basic results about measured asymptotic expanders. These are analogues of Lemma \ref{lem:annoying 1/2 upper bound} and Lemma~\ref{lem:equiv.defn.for.asymp.expansion.domain}(2) respectively. The same proofs work \emph{mutatis mutandis}, and are hence omitted.

\begin{lem}\label{lem:properties.of.measured.asymptotic.expanders}
Let $\{\T_n^{\nu_n}\}_{n\in \NN}$ be a sequence of measured $(\fnc,\fnk)$\=/asymptotic expanders. Then:
\begin{enumerate}
 \item There are functions $\fnb\colon[\frac{1}{2}, 1)\to \RR_{>0}$ and $\fnh\colon[\frac{1}{2}, 1)\to \NN$ depending on $\fnc$ and $\fnk$ such that for every $\beta\in[\frac{1}{2}, 1)$, we have
\[
\mabs{N_{\fnh(\beta)}(A)} > \paren{1+\fnb(\beta)}\mabs{A}
\]
for every $n\in\NN$ and $A \subseteq \T_n^{\nu_n}$ with $\frac{1}{2}\leq \mabs{A} \leq \beta$. If $\fnk$ is constant, $\fnh$ can be taken to be the same constant.
 \item For every $c\in(0,1)$ there exists $\fnk'\colon(0,\frac 12]\to\NN$ such that $\{\T_n^{\nu_n}\}_{n\in \NN}$ is a sequence of measured $(c,\fnk')$\=/asymptotic expanders.
\end{enumerate}
\end{lem}

\begin{de}
 A sequence of probability metric spaces $\braces{\T^{\nu_n}_n}_{n\in\NN}$ has \emph{uniformly bounded measure ratios} if there exists $Q\geq 1$ such that 
 \[
  \frac{1}{Q}\leq\frac{\abs{v}_{\nu_n}}{\abs{w}_{\nu_n}}\leq Q
 \]
 for every pair of points $v,w\in\T^{}_n$ and $n\in\NN$.
\end{de}

 Note that probability metric spaces with uniformly bounded measure ratios must be finite. Also note that in this case, the cardinalities $\abs{\T^{}_n}$ go to infinity if and only if the measures $\abs{v_n}_{\nu_n}$ go to zero for any (hence, every) sequence of points $v_n\in\T^{}_n$. 
 
 The following lemma relates the notion of measured asymptotic expanders with that of the ordinary asymptotic expanders (Definition~\ref{defn:asymptotic.expanders}). 

\begin{lem}\label{lem:measured case vs ordinary case}
 Let $\braces{\T^{\nu_n}_n}_{n\in\NN}$ be a sequence of probability metric spaces with uniformly bounded measure ratios and $\abs{\T^{}_n}\to\infty$. Then $\braces{\T^{\nu_n}_n}_{n\in\NN}$ is a sequence of measured asymptotic expanders \emph{if and only if} their underlying metric spaces $\{\T^{}_n\}_{n\in \NN}$ are a sequence of asymptotic expanders. 
\end{lem}

\begin{proof}
Let $\braces{\T^{\nu_n}_n}_{n\in\NN}$ be a sequence of measured asymptotic expanders and $Q \geq 1$ be the constant witnessing the bound on measure ratios. For any $n\in \NN$ and $A_n\subseteq \T_n$ with $\alpha\abs{\T_n}\leq\abs{A_n}\leq\frac{1}{2}\abs{\T_n}$, we have
\[
\frac{\alpha}{Q^2}\leq |A_n|_{\nu_n}\leq \frac{Q^2}{Q^2+1}.
\]
It hence follows from Lemma \ref{lem:properties.of.measured.asymptotic.expanders}(1) there exist $c$ and $k$ depending only on $\alpha$  such that $|N_k(A_n)|_{\nu_n}\geq(1+c) |A_n|_{\nu_n}$. It is then easy to deduce that $\abs{N_k(A_n)}\geq(1+\frac{c}{Q^2})\abs{A_n}$. The converse implication is analogous, hence omitted.
\end{proof}

\subsection{Measurable partitions and approximating spaces}\label{ssec:partitions.and.approximating.spaces}
A \emph{measurable partition} of a probability space $(X,\nu)$ is a partition of $X$ into a countable family $\P=\{R_i\mid i \in I\}$ of disjoint measurable subsets (regions), \emph{i.e.}, $X=\bigsqcup_{i\in I}R_i$. 

\begin{defn}[{\cite[Definition 3.4]{Vig19}}]\label{defn:bounded measure ratios}
A measurable partition $\P=\{R_i\mid i \in I\}$ has \emph{bounded measure ratios} if there exists a constant $Q\geq 1$ such that for every pair of regions $R_i,R_j$ in $\P$, we have
\[
\frac{1}{Q} \leq \frac{\nu(R_i)}{\nu(R_j)} \leq Q.
\]
\end{defn}

Given any measurable partition $\CP$ of $X$ and a measurable set $A\subseteq X$, we denote by $[A]_\CP$ the set
\[
 [A]_\CP\coloneqq \bigsqcup\bigbraces{R\in\CP\bigmid \nu(R\cap A)>0}\subseteq X
\]
(this notation is inspired by identifying $\CP$ with the equivalence relation given by ``belonging to the same region'', so that $[A]_\CP$ is the \emph{saturation} of $A$ \emph{up to measure zero subsets}). 
Similarly, if $W\subseteq \CP$ is a set of regions of the partition, we denote their union by $[W]_\CP\coloneqq\bigcup\braces{R\mid R\in W}$.

Recall that $\Gamma$ is a countable discrete group equipped with a proper length function $\ell$. One of the main objects of interest in \cite{Vig19} are graphs that approximate measurable actions of finitely generated groups. The following is a (measured) generalisation of that idea to actions of groups equipped with proper length functions:

\begin{defn}\label{defn:approximating graph}
Given a measurable action $\rho\colon \Gamma \act (X,\nu)$ of a countable discrete group $\Gamma$ on a probability space $(X,\nu)$ and a measurable partition $\P=\{R_i\mid i \in I\}$ of $(X,\nu)$, the associated \emph{approximating (metric) space} $\T_\rho(\P)$ is the metric space obtained by equipping $\P$ with the metric $d_{\P}$ defined by 
\[
d_{\P}(R, R')=
\min_{R=R_0,\ldots,R_n=R'}\Big\{ \sum_{i=1}^n \ell(\gamma_i)\bigmid \gamma_i\in \Gamma \mbox{~s.t.~} \nu\big((\gamma_i\cdot R_{i-1}) \cap R_i\big) >0\Big\}\cup\braces{+\infty},
\]
where $n\in\NN$ and $R_i\in \CP$ for every $i=0,\ldots,n$.
Note that $d_{\P}(R, R')$ can never be infinity if the action is ergodic.
When there is no ambiguity, we denote the approximating space by $\T(\P)$.
Note that the approximating space $\T(\P)$ is always discrete since we require that the length function $\ell$ takes values in $\NN \cup \{0\}$.

The approximating space $\T_\rho(\CP)$ comes with a natural Borel measure given by $\mabs{R}\coloneqq\nu(R)$ for every $R\in\CP$. Equipped with this measure, $\T_\rho(\CP)$ is called the \emph{measured approximating (metric) space} associated with $\rho$, denoted by $\T_\rho^\nu(\CP)$. Again when the action is clear from the context, we will simply denote by $\mappgraph{\CP}$.
\end{defn}

The notion of the measured approximating space is designed in such a way that $\mabs{W}=\nu\paren{[W]_\CP}$ for every subset $W\subseteq \CP$. This leads to simple estimates on measured expansion for measured approximating spaces as we will see later.  

\begin{rmk}
If $\Gamma$ is finitely generated by a finite symmetric set $S$ and $\ell$ is the associated word length, the approximating space $\T_\rho(\CP)$ coincides with the approximating graph $\G_\rho(\P)$ (with the edge\=/path metric) defined in \cite[Definition 3.3]{Vig19}.

One of the main results in \cite{Vig19} was to reveal a strong relation between the Cheeger constants of the approximating graphs $\CG_\rho(\CP_n)$ and expansion properties of $\rho$.
The price to pay to compare a purely metric construct (the approximating graph) with a dynamical system is the proliferation of assumptions in the statements of various lemmas and theorems.
In this paper, we prefer to switch from graphs to measured metric spaces because these are better suited to describe approximations of actions on measure spaces.
\end{rmk}

The following lemma is almost immediate:

\begin{lem}\label{lem:asymp.expanding.action.implies.meas.asymp.exp}
Let $\Gamma \act (X,\nu)$ be a $(\fnc,\fnk)$\=/asymptotically expanding action on a probability space $(X,\nu)$. Then for every measurable partition $\P=\{R_i\mid i \in I\}$, the measured approximating space $\mappgraph{\CP}$ is a measured $(\fnc,\fnk)$\=/asymptotic expander.
\end{lem}

\begin{proof}
First note that for every subset $W\subseteq \mappgraph{\CP}$ and $k\in \NN$, we have
\[
B_k\cdot [W]_\CP \subseteq \bigsqcup \big\{R\in\CP \bigmid d_\rho(R,W) \leq k\big\} = [N_k(W)]_\CP
\]
up to measure zero sets.

Given $\alpha \in (0,\frac{1}{2}]$, consider a subset $W\subseteq \mappgraph{\CP}$ with 
$\alpha\leq\mabs{W}\leq\frac{1}{2}$. 
Since $\alpha\leq\nu([W]_\CP)=\mabs{W}\leq \frac{1}{2}$, we deduce from the asymptotic expansion hypothesis that
\[
 \mabs{N_{\fnk(\alpha)}(W)}=\nu\bigparen{[N_{\fnk(\alpha)}(W)]_\CP}
 \geq \nu\bigparen{B_{\fnk(\alpha)}\cdot[W]_\CP}
 > (1+\fnc(\alpha))\nu([W]_\CP)=(1+\fnc(\alpha))\mabs{W}.
\]
Hence $\mappgraph{\CP}$ is a measured $(\fnc,\fnk)$\=/asymptotic expander space.
\end{proof}

\begin{exmp}
 It is proved in \cite[Theorem 11]{ozawa2016remark} that if $G$ is a connected simple Lie group with finite centre, $\Gamma< G$ a lattice and $H<G$ a closed non\=/amenable subgroup then $\Gamma\curvearrowright G/H$ is strongly ergodic (with respect to the measure\=/class induced by the Haar measure of $G$).
 
 If $H<G$ is cocompact and we choose any Riemannian metric on $G/H$, we can easily find a sequence of measurable partition $\CP_n$ with 
 uniformly bounded measure ratios (\emph{e.g.}, by considering Voronoi tessellations as in \cite{Vig19}). It then follows from Lemma~\ref{lem:measured case vs ordinary case} and Lemma~\ref{lem:asymp.expanding.action.implies.meas.asymp.exp} that the associated approximating spaces are asymptotic expanders.
\end{exmp}

\subsection{A converse implication: admissible sequences of partitions}

We will now prove a converse to Lemma~\ref{lem:asymp.expanding.action.implies.meas.asymp.exp}. 
The idea is to consider finer and finer partitions in such a way that we can recognise expansion properties of the action in terms of uniform expansion properties of the associated measured approximating spaces.
The appropriate notion of ``finer and finer partitions'' is somewhat technical, and depends on the action:

\begin{de}\label{defn:admissible.partitions}
 Let $\Gamma\curvearrowright(X,\nu)$ be a measurable action of a countable discrete group $(X,\nu)$ on a probability space $(X,\nu)$ and $k\in\NN$ a natural number. A sequence $\braces{\CP_n}_{n\in\NN}$ of measurable partitions of $X$ is called \emph{$k$\=/admissible (for the action)} if for every measurable subset $A\subseteq X$ and every $n\in\NN$, there is a subset $W_n\subseteq \CP_n$ such that for every $\gamma_1,\ldots,\gamma_q\in \Gamma$ with $\sum_{i=1}^q\ell(\gamma_i) \leq k$, we have
 \begin{equation}\label{eq:admissible.converge.in.measure}
  [\gamma_q\cdot[\;\cdots\,[\gamma_1\cdot[W_n]_{\CP_n}]_{\CP_n}\cdots]_{\CP_n}]_{\CP_n}
  \xrightarrow{\;n\to\infty\;}\, \gamma_q\circ\dots\circ \gamma_1(A)
 \end{equation}
 where the convergence is in measure (\emph{i.e.}, the measure of the symmetric difference goes to zero as $n$ grows to infinity). 
   
A sequence $\braces{\CP_n}_{n\in\NN}$ is \emph{$\infty$\=/admissible} if it is $k$\=/admissible for every $k\in\NN$.
\end{de}

The following is the crucial step to establish the converse to Lemma~\ref{lem:asymp.expanding.action.implies.meas.asymp.exp}:
\begin{prop}\label{prop: meas.asymp.exp implies asymptotic expanding}
 Let $\Gamma\curvearrowright (X,\nu)$ be a measure\=/class\=/preserving action of a countable discrete group $\Gamma$ on a non-atomic probability space $(X,\nu)$, $\fnc_0\colon(0,\frac{1}{2}]\to\RR_{>0}$ and $\fnk\colon(0,\frac{1}{2}]\to\NN$ be fixed functions. Assume that $\braces{\CP_n}_{n\in\NN}$ is a sequence of measurable partitions which is $\fnk(\alpha)$\=/admissible for every $\alpha\in(0,\frac{1}{2}]$. Then the following are equivalent:
 \begin{itemize}
  \item for every function $\fnc<\fnc_0$ (point-wise), $\braces{\mappgraph{\CP_n}}_{n\in\NN}$ is a sequence of measured $(\fnc,\fnk)$\=/asymptotic expanders;
  \item for every function $\fnc<\fnc_0$ (point-wise), the action is $(\fnc,\fnk)$\=/asymptotically expanding in measure.
 \end{itemize} 
\end{prop}

\begin{proof}
The sufficiency is given by Lemma~\ref{lem:asymp.expanding.action.implies.meas.asymp.exp}, so we only focus on the necessity. Fixing a function $\fnc<\fnc_0$ and a parameter $\alpha\in(0,\frac{1}{2}]$, let $k\coloneqq \fnk(\alpha)$ and $c\coloneqq\fnc(\alpha)$.

We begin by fixing a measurable subset $A\subseteq X$ with $\alpha<\nu(A)<\frac{1}{2}$ and letting $W_n\subseteq\CP_n$ be a sequence of subsets from the definition of $k$\=/admissibility. It follows from \eqref{eq:admissible.converge.in.measure} with $\gamma_1=\cdots=\gamma_q=1$ that $[W_n]_{\CP_n}$ converges in measure to $A$. Hence there exists an $n_0$ sufficiently large so that $\alpha<\mabs{W_n}=\nu([W_n]_{\CP_n})<\frac{1}{2}$ for every $n>n_0$. We can hence apply the assumption on $\mappgraph{\CP_n}$ to obtain:
\[
 \mabs{N_{k}(W_n)}>(1+c)\mabs{W_n}.
\]

On the other hand, it follows from the definition of approximating spaces that
\[
 [N_k(W_n)]_{\CP_n}=\bigcup\bigbraces{[\gamma_q\cdot[\;\cdots\,[\gamma_1\cdot[W_n]_{\CP_n}]_{\CP_n}\cdots]_{\CP_n}]_{\CP_n}
  \bigmid 
\gamma_1,\ldots, \gamma_q\in \Gamma \mbox{~s.t.~} \sum_{i=1}^q\ell(\gamma_i) \leq k}.
\]
Hence \eqref{eq:admissible.converge.in.measure} implies that
\[
 [N_k(W_n)]_{\CP_n}\longrightarrow 
 \bigcup\bigbraces{\gamma_q\circ\dots\circ \gamma_1(A) \bigmid \gamma_1,\ldots, \gamma_q\in \Gamma \mbox{~with~} \ell(\gamma_q \cdots\gamma_1) \leq k}
 = B_k\cdot A,
\]
where the convergence is in measure. It follows that
\[
 \nu(B_k\cdot A)=\lim_{n\to\infty}\,\mabs{N_k(W_n)}
 \geq \lim_{n\to\infty}\, (1+c)\mabs{W_n}=(1+c)\nu(A).
\]

It remains to deal with measurable $A\subseteq X$ with measure $\alpha$ or $\frac{1}{2}$. For $\nu(A)=\frac{1}{2}$, $\nu$ being non-atomic implies that there exists an increasing sequence of subsets $A_n\subseteq A$ so that $\nu(A_n)<\frac{1}{2}$ and $\nu(A_n)\to \nu(A)$. For each $n\in \NN$, it follows from the above analysis that $\nu(B_k\cdot A_n)\geq (1+c)\nu(A_n)$. Hence
\begin{equation}\label{eq:prop:limits.of.meas}
 \nu(B_k\cdot A)=\lim_{n\to\infty}\nu(B_k\cdot A_n)\geq \lim_{n\to\infty}(1+c)\nu(A_n)=(1+c)\nu(A),
\end{equation}
where the first equality holds because the action is measure\=/class\=/preserving and $\nu$ is finite. A similar argument works for $\nu(A)=\alpha$ as well. This finishes the proof that $\Gamma\curvearrowright (X,\nu)$ is $(\fnc',\fnk)$\=/asymptotically expanding in measure for every $\fnc'<\fnc$ (it is only necessary to pass to a smaller $\fnc'$ because the definition of asymptotic expansion requires strict inequalities).
\end{proof}

In the general case, we have the following non-quantitative version:

\begin{prop}\label{prop: meas.asymp.exp implies asymptotic expanding general case}
 Let $\Gamma\curvearrowright (X,\nu)$ be a measurable action of a countable discrete group $\Gamma$ on a probability space $(X,\nu)$, and $\braces{\CP_n}_{n\in\NN}$ a sequence of $\infty$\=/admissible measurable partitions. Then the measured approximating spaces $\mappgraph{\CP_n}$ are a sequence of measured asymptotic expanders \emph{if and only if} the action is asymptotically expanding in measure.
\end{prop}
 
\begin{proof}
From Lemma \ref{lem:properties.of.measured.asymptotic.expanders}(1), for any $\alpha\in (0,\frac{1}{2}]$ there exist $c>0$ and $k\in \NN$ such that for any $n\in \NN$ and $A_n \subseteq \T(\CP_n)$ with $\frac{\alpha}{2}< \mabs{A_n} < \frac{2}{3}$, we have $\mabs{N_k(A_n)} > \paren{1+c}\mabs{A_n}$.

Now fix a measurable subset $A\subseteq X$ with $\alpha\leq\nu(A)\leq\frac{1}{2}$, and let $W_n\subseteq\CP_n$ be a sequence of subsets as by the definition of $k$\=/admissibility. It follows from \eqref{eq:admissible.converge.in.measure} that $[W_n]_{\CP_n}$ converges in measure to $A$. Hence there exists an $n_0$ sufficiently large so that $\frac{\alpha}{2}<\mabs{W_n}=\nu([W_n]_{\CP_n})<\frac{2}{3}$ for every $n>n_0$. By construction we have:
\[
 \mabs{N_{k}(W_n)}>(1+c)\mabs{W_n}.
\]
On the other hand, the same argument as in the proof of Proposition \ref{prop: meas.asymp.exp implies asymptotic expanding} shows that
\[
 \nu(B_k\cdot A)=\lim_{n\to\infty}\,\mabs{N_k(W_n)}
 \geq \lim_{n\to\infty}\, (1+c)\mabs{W_n}=(1+c)\nu(A).
\]
So we finish the proof.
 \end{proof}

\begin{rmk}
 From the proof of Propositions~\ref{prop: meas.asymp.exp implies asymptotic expanding} and \ref{prop: meas.asymp.exp implies asymptotic expanding general case}, it is clear that the notion of $k$\=/admissibility was designed to make the argument work.
We deem it a worthwhile effort to explore the precise properties of partitions one need to recover dynamical properties of the action in terms of measured approximating spaces.
\end{rmk}

It now remains to produce examples of $\infty$\=/admissible partitions. For the scope of this paper, the main source of such partitions is the following lemma:

\begin{lem}\label{lem:partitions.with.infinitesimal.mesh.are.admissible}
 Let $\nu$ be a Radon probability 
 measure on a locally compact metric space $(X,d)$ and $\Gamma\curvearrowright X$ a continuous measure\=/class\=/preserving action of a countable discrete group $\Gamma$. 
 If $\braces{\CP_n}_{n\in\NN}$ is a sequence of measurable partitions with ${\rm mesh}(\CP_n)\to 0$ (where ${\rm mesh}(\CP_n)\coloneqq \sup\braces{\diam(R)\mid R\in\CP_n}$), then $\braces{\CP_n}_{n\in\NN}$ is $\infty$\=/admissible for the action.
\end{lem}

\begin{proof}
 Fixing a $k\in\NN$, we begin by showing that we can approximate compact subsets. Let $K\subseteq X$ be a compact subset 
 and set $W_n^K\coloneqq \{R\in \CP_n\mid \nu(R\cap K)>0\}$. Note that $[W_n^K]_{\CP_n}$ is equal to the saturation $[K]_{\CP_n}$.
 
 Let $\delta_n\coloneqq{\rm mesh}(\CP_n)$ and fix a $\gamma\in \Gamma$. 
 For $n$ large enough, $N_{\delta_n}(K)$ is compact and hence the restriction of $\gamma$ to $N_{\delta_n}(K)$ is uniformly continuous. 
 It follows that there exists a sequence $\delta'_n\to 0$ such that $\gamma(N_{\delta_n}(K))\subseteq N_{\delta'_n}(\gamma(K))$.
 Since $[W_n^K]_{\CP_n}\subseteq N_{\delta_n}(K)$, we have $[\gamma\cdot[W_n^K]_{\CP_n}]_{\CP_n}\subseteq N_{\delta_n+\delta'_n}(\gamma(K))$. Since $\gamma(K)\subseteq [\gamma\cdot[W_n^K]_{\CP_n}]_{\CP_n}$ up to measure\=/zero subsets, the sequence $[\gamma\cdot[W_n^K]_{\CP_n}]_{\CP_n}$ converges in measure to $\gamma(K)$. By an inductive argument, the above can be easily extended to show that the sets $W_n^K\subseteq{\CP_n}$ satisfy \eqref{eq:admissible.converge.in.measure} for any $\gamma_1,\ldots,\gamma_q\in \Gamma$ with $\sum_{i=1}^q\ell(\gamma_i) \leq k$.
 
 Now fix a measurable subset $A\subseteq X$.
 Since $\nu$ is Radon, there is a sequence of compact sets $K_m\subseteq A$ that converges to $A$ in measure. Each of these compact sets come with its own ``approximating sequence'' $(W_n^{K_m})_{n\in\NN}$. 
 A straightforward diagonal argument is thus sufficient to produce a sequence $(W_n)_{n\in\NN}$ that satisfies \eqref{eq:admissible.converge.in.measure} for $A$.
\end{proof}

Together with Proposition~\ref{prop: meas.asymp.exp implies asymptotic expanding general case}, we obtain the following:

\begin{thm}\label{thm: asymptotic expansion and approx spaces}
Let $(X,d,\nu)$ be a locally compact metric space with a Radon probability measure, and $\{\P_n\}_{n\in \NN}$ be a sequence of measurable partitions of $X$ with $\mathrm{mesh}(\P_n) \to 0$.
Then any continuous measure\=/class\=/preserving action $\Gamma \act X$ of a countable discrete group $\Gamma$ is asymptotically expanding in measure \emph{if and only if} the measured approximating spaces $\braces{\mappgraph{\P_n}}_{n \in \N}$ are a sequence of measured asymptotic expanders.
\end{thm}

Assuming bounded measure ratios, we can apply Lemma \ref{lem:measured case vs ordinary case} to obtain an analogous characterisation in terms of (non-measured) asymptotic expanders:

\begin{cor}\label{cor: asymptotic expansion and approx spaces}
Let $(X,d,\nu)$ and $\{\CP_n\}_{n\in \NN}$ be as in Theorem~\ref{thm: asymptotic expansion and approx spaces} with $\{\CP_n\}_n$ having uniformly bounded measure ratios. Then a continuous measure\=/class\=/preserving action $\Gamma \act X$ is asymptotically expanding in measure \emph{if and only if} the approximating spaces $\braces{\appgraph{\P_n}}_{n \in \N}$ are a sequence of asymptotic expanders.
\end{cor}

One source of examples where the above result can be applied is given by boundary/profinite actions as described in \cite{abert2012dynamical}. Namely, if $\Gamma_1>\Gamma_2>\cdots $ is a chain of finite index subgroups of $\Gamma$, then the boundary of the induced coset tree is a compact totally disconnected topological space $X$ equipped with a natural probability measure $\nu$. The group $\Gamma$ acts on $X$ and the action preserves $\nu$. In the notation of \cite{abert2012dynamical}, we can let $\CP_n\coloneqq\braces{{\rm Sh}(\Gamma_ng)\mid \Gamma_ng\in \Gamma_n\bs\Gamma}$. These partitions have measure ratios $\equiv 1$. Furthermore, it is also possible to equip $X$ with a profinite metric and ${\rm mesh}(\CP_n)\to 0$ with respect to this metric. It follows from Corollary~\ref{cor: asymptotic expansion and approx spaces} that the associated approximating graphs are asymptotic expanders if and only if the boundary action $\Gamma\curvearrowright(X,\nu)$ is strongly ergodic (note that in this context the approximating graphs are nothing but the Schreier graphs). In \cite[Theorem 5]{abert2012dynamical} it is also shown that there exist chains $\Gamma_n$ such that the associated boundary action is strongly ergodic but does not have spectral gap. It follows that the associated approximating graphs are asymptotic expanders that are not expanders.

\

The following example shows that Corollary~\ref{cor: asymptotic expansion and approx spaces} fails when considering non-continuous (measure\=/preserving) actions on spaces with Radon measures.

\begin{example}\label{exmp: nasty non admissible}
 Consider the unit interval $[0,1]$ with the Lebesgue measure $\lambda$ (defined on the Borel $\sigma$\=/algebra). Let $A_0\coloneqq [0,1]$ and choose a sequence $(A_n)_{n\geq 1}$ of measurable subsets of $[0,1]$ such that $\lambda(A_n)=n4^{-n}$ and $\lambda(A_n\cap [\frac{k-1}{n},\frac{k}{n}])= 4^{-n}$ for every $k=1,\ldots, n$. Let $B_n\coloneqq A_n\smallsetminus\paren{\bigcup_{m>n}A_m}$. The $B_n$ are a partition of $[0,1]$ into a sequence of disjoint subsets which become ``smaller and smaller'' but ``denser and denser''.
 
 By integrating the indicator functions, we obtain (up to measure\=/zero) a measure\=/preserving bijection $f\colon [0,1]\to [0,1]$ sending $B_0$ to $[0,\lambda(B_0)]$, $B_1$ to $[\lambda(B_0),\lambda(B_0)+\lambda(B_1)]$ and so on. 
 We can now consider the induced $\ZZ$\=/action on $([0,1],\lambda)$.
 
 Let $\CP_n\coloneqq\braces{[\frac{k-1}{n},\frac{k}{n}]\mid k=1,\ldots, n}$. Then the (measured) approximating spaces $\T^\lambda(\CP_n)$ are measured $(1,2)$\=/asymptotic expanders (independently of $\alpha$). In fact, the region $[\frac{n-1}{n},1]$ intersects with positive measure $f(R)$ for every $R\in\CP_n$. It follows that $N_2(W_n)=\CP_n$ for every non\=/empty subset $W_n\subseteq\CP_n$.
 
 One can show that sequence of partitions $\braces{\CP_n}_{n\in\NN}$ is $1$\=/admissible. 
 On the other hand, by looking at subsets of the form $\bigcup_{k=1}^N f^k(C)$ with suitably small $C\subseteq [0,1]$, it is easy to check that the action $\ZZ\curvearrowright ([0,1],\lambda)$ is \emph{not} asymptotically expanding.
 In particular, it follows from Proposition~\ref{prop: meas.asymp.exp implies asymptotic expanding} that $\braces{\CP_n}_{n\in\NN}$ is \emph{not} 2\=/admissible (it is also easy to check it by hand). 
 More interestingly, this example shows that the hypothesis of $k$\=/admissibility is sharp.
\end{example}

\begin{rmk}
 It is not clear to us if one can produce examples of actions and partitions that are $2$\=/admissible without being $\infty$\=/admissible.
\end{rmk}

\bibliographystyle{plain}
\bibliography{bibfile,ExpanderishVig}

\end{document}